\documentclass[11pt]{amsart}
\usepackage{amsmath,latexsym,amssymb,times,mathptm,enumerate,amsbsy,amsrefs,  hyperref, color, graphicx}
\hypersetup{colorlinks,  citecolor=blue, linkcolor=blue, urlcolor=blue}
\usepackage[all]{xy}

\newenvironment{pf}{\proof[\proofname]}{\endproof}

\theoremstyle{plain}
\newtheorem{thm}{Theorem}[section]
\newtheorem{lem}[thm]{Lemma}

\newtheorem{prop}[thm]{Proposition}
\newtheorem{cor}[thm]{Corollary}
 \theoremstyle{definition}
\newtheorem{df}[thm]{Definition}

\newtheorem{ex}[thm]{Example}
\newtheorem{rem}[thm]{Remark}

\def\m{{\mathfrak m}}

\def\Z{{\mathbb Z}}
\def\N{{\mathbb N}}
\def\R{{\mathbb R}}

\def\Conv{{\rm conv}}
\def\Vol{{\rm Vol}}
\def\vol{{\rm vol}}

\def\Pyr{{\rm pyr}}
\def\dim{{\rm dim}\, }
\def\rank{{\rm rank}\, }

\def\eps{\varepsilon}

\usepackage{mathtools}

\newcommand{\cal}[1]{\mathcal{#1}}

\newcommand{\Q}{\mathbb Q}

\newcommand{\G}{\Gamma}

\newcommand{\D}{\Delta}

\newcommand{\cB}{\cal B}

\newcommand{\cG}{\cal G}

\newcommand{\cR}{\cal R}

\newcommand{\cl}{\operatorname{cl}}
\newcommand{\rk}{\operatorname{rank}}

\newcommand{\rs}[1]{Section~\ref{S:#1}}
\newcommand{\rl}[1]{Lemma~\ref{L:#1}}
\newcommand{\rp}[1]{Proposition~\ref{P:#1}}

\newcommand{\rr}[1]{Remark~\ref{R:#1}}
\newcommand{\rex}[1]{Example~\ref{ex:#1}}
\newcommand{\re}[1]{(\ref{e:#1})}
\newcommand{\rc}[1]{Corollary~\ref{C:#1}}
\newcommand{\rt}[1] {Theorem~\ref{T:#1}}

\newcommand{\rf}[1]{Figure~\ref{F:#1}}


\begin{document}
\title{Generalized multiplicities of edge ideals}

\author{Ali Alilooee}
\address{Department of Mathematics \\ Western Illinois University \\ Macomb, IL 61455}
\email{a-alilooeedolatabad@wiu.edu}

\author{Ivan Soprunov}
\address{Department of Mathematics\\ Cleveland State University\\ Cleveland, OH 44115}
\email{i.soprunov@csuohio.edu}

\author{Javid Validashti}
\address{Department of Mathematics\\ Cleveland State University \\ Cleveland, OH 44115}
 \email{j.validashti@csuohio.edu}

\begin{abstract}
We explore connections between the generalized multiplicities of square-free monomial ideals and the combinatorial structure of the underlying hypergraphs using methods of commutative algebra and polyhedral geometry. For instance, we show the $j$-multiplicity is multiplicative over the connected components of a hypergraph, and we explicitly relate the $j$-multiplicity of the edge ideal of a properly connected uniform hypergraph to the Hilbert-Samuel multiplicity of its special fiber ring. In addition, we provide general bounds for the generalized multiplicities of the edge ideals and compute these invariants for classes of uniform hypergraphs.
\end{abstract}

\maketitle


\section{Introduction}

The theory of multiplicities is centuries old and it involves a rich interplay of ideas from various fields, including algebraic geometry, commutative algebra, convex geometry, and combinatorics. The first rigorous general algebraic treatment of multiplicities was given by Chevalley and Samuel  for zero-dimensional ideals \cite{Chevalley1, Chevalley2, Samuel1, Samuel2} and soon they became ubiquitous in commutative algebra. For instance, the Hilbert-Samuel multiplicity plays a prominent  role in the theory of integral dependence of ideals due to the influential work of Rees \cite{R1}.
Multiplicity theory has also close ties with polyhedral geometry via Ehrhart theory. 
In addition, the Hilbert-Samuel multiplicity of zero-dimensional monomial ideals  has  an elegant interpretation in convex geometry and combinatorics. Indeed, the multiplicity of a zero-dimensional monomial ideal  is equal to the normalized  full-dimensional volume of the complement of its Newton polyhedron in the positive orthant \cite{Tei}. More recently, Achilles and Manaresi introduced the concept of $j$-multiplicity \cite{AM}, and Ulrich and Validashti proposed the notion of $\eps$-multiplicity \cite{UV},  extending the  classical Hilbert-Samuel multiplicity to arbitrary ideals in a general algebraic setting. These invariants have been proven  useful in commutative algebra and algebraic geometry for their connections to the theory of integral closures and Rees valuations, the study of the associated graded algebras, intersection theory, equisingularity and local volumes of divisors \cite{FM, KV, MX, PX, UV}. Recently, Jeffries and Monta{\~n}o showed that these numbers measure certain volumes defined for arbitrary monomial ideals, similar to the zero-dimensional case \cite{JM}. Currently, there is  a rising interest in finding formulas for the $j$-multiplicity of classes of  ideals \cite{NU, JMV}. The main objective of this paper is to further understand how the $j$-multiplicity and the $\eps$-multiplicity manifest in various combinatorial structures and invariants. In particular, we consider square-free monomial ideals associated to hypergraphs, called the edge ideals,  which are not zero-dimensional, and we explore connections between the generalized multiplicities of such ideals and the combinatorial  properties of the underlying hypergraphs. It is notable that  \cite[Theorem 3.2]{JM} implies that the $j$-multiplicity of the edge ideal of a uniform hypergraph and the normalized volume of the associated edge polytope are the same up to a constant factor. Thus, the theory of $j$-multiplicity in particular provides a new perspective on the edge polytopes which may contribute to the currently limited information about these objects, and vice versa.  Geometric features of edge polytopes as well as algebraic properties and invariants of the edge ideals such as regularity, Cohen-Macaulayness, their symbolic Rees algebras and core have been studied extensively in commutative algebra and combinatorics \cite{MV, OH1998, SVV, Vill2, Vill1, V1998, Ziegler}. Our main results concerning the generalized multiplicities of the edge ideals are the following. \\

 Let $G$ be a hypergraph on $n$ nodes with edge ideal $I(G)$ and Newton polyhedron $P(G)$.  We show that
 the normalized volume is multiplicative with respect to free sums of co-convex sets (\rp{freesums}) which produces a multiplicativity formula for the $j$-multiplicity for monomial ideals (\rt{product}). In particular, if $G_1,\dots, G_c$ are the connected components of $G$, then  we obtain $j(I(G))=j(I(G_1))\cdots j(I(G_c))$ (\rp{product}), but this relation is not true for the $\eps$-multiplicity (\rr{epsilonremarks}). Assume each connected component  of $G$ is properly connected. Then we observe the analytic spread of $I(G)$ equals $n-p+c$,  where $p$ is the number of the node pivot equivalence classes of $G$ (\rp{HyperSpread}). In particular, this implies the $j$-multiplicity and the $\eps$-multiplicity of the edge ideal of $G$ are not zero if and only if the nodes in each connected component of $G$ are pivot equivalent (\rp{uniformpositivity}). In this case, we prove that $j(I(G))= m^c e(k[G])$, where $e(k[G])$ is the Hilbert-Samuel multiplicity of the edge subring $k[G]$ (\rt{cedgesubring}). 
 As an application, we obtain a formula relating the Hilbert-Samuel multiplicity of the edge subring of $G$
 to the volume of its edge polytope (\rc{edgepolytope}).
 Moreover, we note that the height of the toric edge ideal of $G$  is $e-n+p-c$, where $e$ is the number of edges in $G$ (\rp{heighttoric}). As an application we obtain  the following when $j(I(G))$ is not zero:  If $e=n$ then $j(I(G))= m^c$(\rp{muniformunicyclic}), and  if $e=n+1$ then $j(I(G))=m^cl$, where $l$ is half the length of the unique nontrivial minimal monomial walk in $G$ up to equivalence (\rp{bicyclic}).  We also prove  $j(I(G))$ is greater than or equal to $j(I(H))$   for any subhypergraph $H$ of $G$, provided $j(I(G))$ is not zero (\rt{inequality}), and equality holds when $H$ is obtained from $G$ by removing a free node (\rp{whisker}).
These statements fail to be true for the $\eps$-multiplicity (\rr{epsilonremarks}).  As a corollary we conclude $j(I(G))$ is bounded above the $j$-multiplicity of the complete $m$-uniform hypergraph on $n$ nodes as in \rex{CompleteUniform}. In particular, if $G$ is a simple graph on $n$ nodes such that $j(I(G))$ is not zero, then $j(I(G))$ is between $2^{\tau_0}$ and $2^n -2n$, where $\tau_0$ is the odd tulgeity of $G$ (\rc{bounds}). In addition, we show that if $G$ is an odd cycle of length $n$, then $\eps(I(G))=\frac{2}{n+1}$ (\rp{epsiloncycle}) and we compute the $\eps$-multiplicity of the edge ideals of complete $m$-uniform hypergraphs (\rp{epsiloncomplete}). Throughout  the paper, we develop results from the perspective of both commutative algebra and polyhedral geometry which reveals a beautiful interaction of ideas between the two approaches.
\\

The paper is organized as follows. 
In \rs{jmult} we review the notion of $j$-multiplicity in a general algebraic setting.
In \rs{jmultVolume} we recall the  connection between the $j$-multiplicity of monomial ideals and the associated  polytopes. 
In \rs{sectionfreesum} we describe a connection between  the $j$-multiplicity and the free sum of co-convex sets
and prove the multiplicativity of the  $j$-multiplicity of edge ideals over the connected components. 
In \rs{edgeidealvolumes} we  further explore the $j$-multiplicity of edge ideals via volumes.
In \rs{pecas} we give a formula for the analytic spread of edge ideals and we obtain a combinatorial characterization of the vanishing of their $j$-multiplicity and $\eps$-multiplicity using pivot equivalence relation.
In \rs{edgeidealedgesubring} we study the relation between the $j$-multiplicity of the edge ideal of a hypergraph and the associated edge subring.
In \rs{toricedgeideal} we use toric edge ideals to obtain a formula for the $j$-multiplicity of the edge ideal of classes of hypergraphs.
In \rs{inequalityedgeideal} we provide general bounds for the $j$-multiplicity of edge ideals.
In \rs{epsilonedgeideal} we compute the $\eps$-multiplicity of the edge ideals of cycles and complete hypergraphs.

\section{The $j$-multiplicity}\label{S:jmult}

Let  $R$ be a Noetherian local ring with  maximal ideal $\m$ and Krull dimension $n$. We recall the notion of $j$-multiplicity $j(I)$ of an ideal $I$ in $R$ as introduced and developed in \cite[6.1]{FOV} and \cite{AM}. Let $S$ be a standard graded Noetherian $R$-algebra, that is, a graded $R$-algebra with $S_{0}=R$ and generated by finitely many homogeneous elements of degree one.  Then $\Gamma_{\m}(S) \subset S$ is a graded ideal in $S$, where $\Gamma_{\m}$ denotes the zeroth local cohomology with respect to the ideal $\m$ of $R$. In particular, $\Gamma_{\m} (S)$ is finitely generated over $S$. Thus there exists a fixed power ${\m}^t$ of $\m$ that annihilates $\Gamma_{\m} (S)$ . Therefore $\Gamma_{\m}(S)$ is a finitely generated graded module over $S/{\m}^tS$, which is a standard graded Noetherian algebra over the
Artinian local ring $R/{\m}^t$. Hence $\Gamma_{\m}(S)$ has a Hilbert function that is eventually polynomial of degree at most $\dim S -1$, whose normalized leading coefficient is the Hilbert-Samuel multiplicity $e(\Gamma_{\m}(S))$. We define the $j$-multiplicity $j(S)$  to be $e(\Gamma_{\m}(S))$ when $\dim \Gamma_{\m}(S) = \dim S$ and zero otherwise. If $S_k$ is the graded component of $S$ of degree $k$ and $\lambda$ denotes the length, we may write
$$ 
j(S)= (\dim  S-1)! \lim_{k \to \infty}  \frac{\lambda _R(\Gamma_{\m}(S_k))}{k^{\dim S-1}}.
$$
If the graded components of $S$ have finite length, then   $j(S)$ is the same as the Hilbert-Samuel multiplicity $e(S)$.
In addition, one can see that  the condition $\dim \Gamma_{\m} (S) < \dim S$ is equivalent to $\dim S/ \m S < \dim S$. Therefore, one has 
\begin{rem}\label{R:gvanishing}
$j(S)=0$ if and only if $\dim  S/ \m S <  \dim  S$. 
\end{rem}


Recall that  the associated graded ring  of $R$ with respect to an ideal $I$, which we denote by $\cG$,  is a standard graded Noetherian $R/I$-algebra of dimension $n$. Then, the  $j$-multiplicity $j(I)$ is defined as the $j$-multiplicity of the  graded ring $\cG$. In terms of the length of the graded components of $\Gamma_{\m}(\cG)$ we may write
$$ 
j(I)= (n-1)! \lim_{k \to \infty}  \frac{\lambda _R(\Gamma_{\m}(I^k/I^{k+1}))}{k^{n-1}}.
$$
If $I$ is $\m$-primary, then the graded components of the associated graded ring of $R$ with respect to $I$ have finite length, and  $j(I)$ is indeed the Hilbert-Samuel multiplicity $e(I)$.
Moreover, $j(I)=0$ if and only if $\dim  \cG/ \m \cG <  \dim  \cG=n$ by \rr{gvanishing}. The dimension of  the special fiber ring $\cG/ \m \cG $ is denoted by $\ell(I)$ and is called the analytic spread of $I$. Thus, we have
\begin{rem}\label{R:jvanishing}
$j(I)=0$ if and only if $\ell(I)< n$.
\end{rem}
We refer the reader to \cite{FOV} for further properties of $j$-multiplicities, and to  \cite{BH} for unexplained terminology.

\section{The $j$-multiplicity of monomial ideals and volumes}\label{S:jmultVolume}

We begin with recalling some definitions and notation from convex geometry related to monomial ideals. 
Consider the integer lattice $\Z^n$ in $\R^n$. A {\it lattice polytope} $F$ in $\R^n$
is the convex hull of finitely many lattice points. A {\it unimodular $n$-simplex} is the convex hull of
$n+1$ lattice points $\{v_0,v_1,\dots, v_n\}$ such that $\{v_1-v_0,\dots, v_n-v_0\}$ is
a basis for the lattice. 
We use $\Vol_n$ to denote the {\it normalized $n$-dimensional volume} in $\R^n$ defined such
that $\Vol_n(\D)=1$ for any unimodular $n$-simplex $\D$. Then for any lattice polytope $F$ we have
$\Vol_n(F)=n!\vol_n(F)$, where $\vol_n$ is the usual Euclidean volume in~$\R^n$. Similarly,
we can define the normalized $k$-dimensional volume with respect to any sublattice in $\Z^n$
of rank $k$. We will be concerned with the following particular situation. Suppose $F$ is a lattice polytope lying in a rational affine hyperplane 
$$L=\{z\in\R^n \ | \ \langle u,z\rangle=b\},$$
where $b\in\Z$, $b\geq 0$, and $u=(u_1,\dots,u_n)$ is a primitive integer vector, that is $\gcd(u_1,\dots, u_n)=1$. We use  $\langle u,z\rangle$ to denote the inner product of $u$ and $z$ in $\R^n$.
Then we write $\Vol_{n-1}(F)$ to denote the normalized  $(n-1)$-dimensional volume with respect to the sublattice $L\cap\Z^n\subset\Z^n$.  Note that the integer $b$ is the {\it lattice distance} from  $L$ to the origin. For a lattice polytope $F\subset\R^n$ of dimension at most $n-1$, we write $\Pyr(F)$ for  the convex hull of $F$ and the origin, 
which we call the pyramid over~$F$. 
Clearly, $\Vol_n(\Pyr(F))=0$ if $\dim F$ is less than $n-1$. When $\dim F=n-1$ we have
the following formula  which is  standard in lattice geometry:
\begin{equation}\label{e:pyrvolume}
\Vol_n(\Pyr(F))=h(F)\Vol_{n-1}(F),
\end{equation}
where $h(F)$ is the lattice distance from the affine span of $F$ to the origin. More generally,
let $a\in\Q^n$ be such that $\langle u,a\rangle\leq b$. Then the convex hull $\Pyr_a(F)$ of $a$ and $F$ is the
pyramid over $F$ with apex $a$ and lattice height $h(F)-\langle u,a\rangle$. Therefore we obtain
\begin{equation}\label{e:pyrvolume-apex}
\Vol_n(\Pyr_a(F))=(h(F)-\langle u,a\rangle)\Vol_{n-1}(F).
\end{equation} 
Here $h(F)-\langle u,a\rangle$ is the {\it lattice distance} from the affine span of $F$ to $a$.

Now let  $I$ be a monomial ideal in  $R=k[x_1,\dots, x_n]_{(x_1,\dots, x_n)}$.  
The {\it Newton polytope $F(I)$} is the convex hull in $\R^n$ of the exponent vectors of the minimal generators of~$I$, and the {\it Newton polyhedron $P(I)$} is the convex hull in $\R^n$ of the exponent vectors of all monomials in $I$. 
The following result due to Jeffries and Monta{\~n}o  \cite[Theorem 3.2]{JM} relates the $j$-multiplicity of a monomial ideal to the underlying Newton polyhedron.
\begin{thm}\label{T:JonathanJack}
Let $I$ be a monomial ideal and $F_1,\dots, F_k$ be the compact facets of $P(I)$. 
Then $$j(I)= \sum_{j=1}^k\Vol_n(\Pyr(F_j)) =  \sum_{j=1}^k h(F_j) \Vol_{n-1}(F_j),$$
where $h(F_j)$ is the lattice distance from the affine span of $F_j$ to the origin.
\end{thm}
Recall that by  \rr{jvanishing}, $j(I)=0$ if and only if $\ell(I)$ is less than $n$. On the other hand,
by a result of Bivi{\`a}-Ausina \cite{AB}, the analytic spread of $I$ is  the maximum of the dimensions of the compact faces of $P(I)$ plus one. Therefore, we obtain
\begin{rem}\label{R:Bivia}
$j(I)=0$ if and only if all compact faces of $P(I)$ have dimension less than $n-1$, that is $P(I)$ has no compact facets.
\end{rem}

\begin{ex}\label{ex:CompleteUniform}
Let $I$ be the ideal generated by all square-free monomials of degree $m$ in $R$. Then, the Newton polytope of $I$ is the convex hull of all vectors in $\R^n$ with exactly $m$ entries being $1$ and the rest $0$. Therefore,  $I$ corresponds to a hypersimplex of type $(m, n)$  lying in the hyperplane $z_1+\cdots+z_n=m$. It is classical that ${\Vol_{n-1}}(F(I))$ equals the  Eulerian number $A(n-1,m)$. 
Therefore, by \rt{JonathanJack} we obtain a closed formula
$$
j(I)= m \cdot A(n-1,m) = m \cdot \left( \sum_{k=0}^{m} (-1)^k{n \choose k} (m-k)^{n-1} \right).
$$
For instance, if $m=2$ then $j(I)=2^n -2n$, and if $m=n-1$ then $j(I)=n-1$. Note that $j(I)=0$
if and only if $m=n$.

\end{ex}

Below we provide a simple proof of \rt{JonathanJack} when $I$ is a monomial ideal of the form $wJ$, where $w$ is a monomial and $J$ is a zero-dimensional monomial ideal in $R$, using the volume interpretation of the Hilbert-Samuel multiplicity of  zero-dimensional monomial ideals due to Teissier \cite{Tei}. Note that all monomial ideals of a polynomial ring in two variables are of form $wJ$ as above.
\begin{proof}
First note that by Theorem \cite[3.12]{KV}, $j(I) = j(wJ) = e(J) + e(J\bar{R})$, where $\bar{R}=R/(w)$.
Write $w$ as  $x_1^{a_1} \cdots x_n^{a_n}$. By the associativity formula for the Hilbert-Samuel multiplicity,
$$
e(J\bar{R}) = \sum_{i=1}^n \lambda ((\bar{R})_{(x_i)})\cdot e(J(\bar{R}/x_i\bar{R})) =\sum_{i=1}^n a_i \cdot e(J{R}_i)
$$
where ${R}_i=R/(x_i)$. Hence we obtain
\begin{equation}\label{e:associativity}
j(I) = e(J) +  \sum_{i=1}^n a_i \cdot e(J{R}_i).
\end{equation}
For a polyhedron $P$ denote by $c(P)$ the union of the pyramids over the compact faces of $P$.
Using Teissier's result for the zero-dimensional ideal $J$  we have $e(J)=\Vol_n(c({P(J)}))$.
For $i=1, \ldots, n$, let $P_i$ be the facet of $P(J)$ with the inner normal vector $e_i$.
Then $P_i$ is the Newton polyhedron of the zero-dimensional ideal $J{R}_i$ and, hence,  $e(J{R}_i)=\Vol_{n-1}(c({P_i}))$, again by Teissier's result. Therefore,
$$
j(I) = \Vol_n(c({P(J)})) +  \sum_{i=1}^n a_i \Vol_{n-1}(c({P_i})).
$$
We claim that the latter equals $\Vol_n(c({P(I)}))$. Note that
$P(I)=P(J)+a$, where $a=(a_1,\dots, a_n)$ as above. Let  $F_j$ be the compact facets of $P(J)$ 
with primitive inner normals $\eta_j\in\Z^n$, for $1\leq j\leq k$. As the compact facets of $P(I)$ are translates of the $F_j$ we have
\begin{equation}\label{e:volumes}
\Vol_n(c({P(I)}))=\sum_{j=1}^k\min_{u\in P(I)}\langle u,\eta_j\rangle\Vol_{n-1}(F_j)=
\sum_{j=1}^k\min_{u\in P(J)}\langle u,\eta_j\rangle\Vol_{n-1}(F_j)+\sum_{j=1}^k\langle a,\eta_j\rangle\Vol_{n-1}(F_j).
\end{equation}
The first summand in the right hand side of \re{volumes} equals $\Vol_n(c({P(J)}))$. For the second summand we have
\begin{equation}\label{e:volumes2}
\sum_{j=1}^k\langle a,\eta_j\rangle\Vol_{n-1}(F_j)=\sum_{i=1}^na_i\sum_{j=1}^k\langle e_i,\eta_j\rangle\Vol_{n-1}(F_j).
\end{equation}
\rl{projection} below implies
that the projection of the union of the $F_j$ onto $L_i$ gives a polyhedral subdivision of
$c({P_i})$. As the projection of $F_j$ onto $L_i$ has volume 
$\langle e_i,\eta_j\rangle\Vol_{n-1}(F_j)$, we get
$$\Vol_{n-1}(c({P_i}))=\sum_{j=1}^k\langle e_i,\eta_j\rangle\Vol_{n-1}(F_j).$$
Combining this with \re{volumes2} and \re{volumes} we obtain
$$\Vol_n(c({P(I)}))=\Vol_n(c({P(J)})) +  \sum_{i=1}^n a_i \Vol_{n-1}(c({P_i})),$$
as claimed.

%
%
\end{proof}

\begin{lem}\label{L:projection}
Let $P$ be a polyhedron in the $n$-orthant $\R^n_{\geq 0}$
whose complement $\R^n_{\geq 0}\setminus P$ is bounded. Let $L_i=\{z\in \R^n\ |\ z_i=0\}$
be a coordinate hyperplane. Then the projection $\pi_i:\R^n\to L_i$
gives a bijection between the union of  the compact
facets of $P$ and the closure of the complement of $P\cap L_i$ in the 
$(n-1)$-orthant $\R^n_{\geq 0}\cap L_i$.
\end{lem}

\begin{proof} 
First note that the non-compact facets of $P$ are precisely the intersections $P\cap L_i$ for $1\leq i\leq n$. This implies that the union of the compact facets $\cal F$ of $P$ equals the closure of 
$\partial P\cap \R^n_{>0}$.  
In addition, the inner normals of the compact facets of $P$ have all their coordinates positive. 
To simplify notation we assume $i=n$ and let 
$P'=P\cap L_n$ and $c({P'})$ be the closure of the complement of $P'$ in 
$\R^n_{\geq 0}\cap L_n$.

First we check that  $\pi_n$ restricted to ${\cal F}$ is one-to-one. Indeed,
suppose $a_1=(a',t_1)$ and $a_2=(a',t_2)$ lie in ${\cal F}$
for some $(a',0)\in L_n$ and $t_1,t_2\geq 0$ and assume $t_1\leq t_2$.
Let $\eta$ be an inner normal to a facet containing $a_2$. Then $\langle \eta,z\rangle$ attains its minimum on $P$ at $z=a_2$, but since $a_1\in P$ and $\eta_n>0$ we must have $t_2\leq t_1$. Therefore, $t_1=t_2$
and so $a_1=a_2$.

Now we show that $\pi_n({\cal F})=c({P'})$.
Let $a_0=(a',0)$ be an interior point of $c({P'})$ (relative to $L_n$) and thus $a_0\not\in P$.
Since $\R^n_{\geq 0}\setminus P$ is bounded, 
$(a',t)\in P$ for $t\gg 0$. Since $P$ is closed, there exists the smallest
value of $t>0$ such that $a=(a',t)$ lies in $P$ and, hence, in the boundary of $P$. Thus, $a$ lies in a compact facet of $P$, as all coordinates of $a$ are positive. 
Therefore the interior of $c({P'})$ is contained in $\pi_n({\cal F})$. 
Since ${\cal F}$ is closed, by continuity, $c({P'})\subseteq\pi_n({\cal F})$.
Finally, if $\pi_n(a)=(a',0)\in P'$ for some $a=(a',t_1)\in{\cal F}$
then the entire ray $\{(a',t)\ |\ t\geq 0\}$ lies in $P$. By the same argument as in the previous paragraph
we must have $t_1\leq 0$, thus $t_1=0$. In other words, $\pi_n(a)=a$ lies in the boundary of $P'$.
Therefore,  $\pi_n({\cal F})\subseteq c({P'})$.

\end{proof}

\section{The $j$-multiplicity of monomial ideals and free sums}\label{S:sectionfreesum}

In this section we observe that if $I$ is a sum of monomial ideals  whose sets of minimal monomial generators
involve pairwise disjoint collections of variables, then the $j$-multiplicity of $I$ is the product of the $j$-multiplicities of the summands, see \rt{product}.
The combinatorial counterpart here is the  free sum of co-convex bodies.

Recall the notion of a co-convex body. Let $C\subset \R^n$ be a closed convex cone with non-empty interior which does not contain non-trivial linear subspaces.
Let $P\subset C$ be a convex set such that $C\setminus P$ is bounded. Then the closure of $C\setminus P$, denoted by $c(P)$, is called a {\it co-convex} body. Furthermore, let $F(P)=c(P)\cap P$ which is the union of the bounded faces of $P$.
For example, let $F(I)$ be the Newton polytope and $P(I)$ be the Newton polyhedron of a monomial ideal $I$ in  $R=k[x_1,\dots, x_n]_{(x_1,\dots, x_n)}$. Let $C$ be the cone over $F(I)$ and $P=P(I)\cap C$. Then the co-convex body $c(P)$
is the union of pyramids over the bounded faces of $P(I)$. Its normalized volume equals the $j$-multiplicity of the ideal $I$
\begin{equation}\label{e:coconvex}
j(I)=\Vol_n(c(P)),
\end{equation}
according to \rt{JonathanJack}.

\begin{df} Let $P_i\subset C_i\subset \R^{n_i}$, for $i=1,2$, be convex sets contained
in convex cones as above and $K_i=c(P_i)$ the corresponding co-convex bodies. Define
the {\it free sum}  $P_1\oplus P_2$ to be the convex hull of the union 
$(P_1\times\{0\})\cup(\{0\}\times P_2)$ in $\R^{n_1}\times \R^{n_2}$. The closure of the complement of
$P_1\oplus P_2$ in $C_1\times C_2$ is called the {\it free sum} of the co-convex bodies
$K_1$ and $K_2$, and is denoted by $K_1\oplus K_2$.
\end{df}


\begin{ex}\label{ex:simplex} 
Let $\D_1$ be an $n_1$-simplex generated by integer vectors $v_1,\dots, v_{n_1}$ in $\R^{n_1}$
and $\D_2$ be an $n_2$-simplex generated by integer vectors $w_1,\dots, w_{n_2}$ in 
$\R^{n_2}$ and let $n=n_1+n_2$. Then $\D_1\oplus \D_2$ is the $n$-simplex generated by $(v_1,0),\dots, (v_{n_1},0),(0,w_1),\dots, (0,w_{n_2})$. Moreover, the normalized volumes of $\D_1$, $\D_2$, and $\D_1\oplus \D_2$ satisfy
$$\Vol_n(\D_1\oplus \D_2)=\Vol_k(\D_1)\Vol_l(\D_2).$$
Indeed, the volume on the left equals the absolute value of the determinant of the block matrix
with blocks corresponding to the two sets of vectors. 
\end{ex}

The above property about normalized volumes extends to free sums of arbitrary convex sets
containing the origin, as well as to co-convex bodies.
For convex centrally symmetric bodies this follows from \cite[p. 15]{RyaZva} but the argument can be  adapted  to the case of co-convex bodies as sketched  below. A different proof for convex sets containing the origin was found by T.~McAllister (private communication).

Let $K=c(P)\subset C$ be a co-convex body. The {\it Minkowski functional} of $K$ is defined on $C$ by 
$$|x|_K=\inf\{r\geq 0\ |\ x\in rK\}.$$
Note that $K$ is the set of those $x\in C$ with $|x|_K\leq 1$ and 
$F(P)$ is the set of  $x\in C$ with $|x|_K=1$. Furthermore, for any $r\geq 0$,
the dilation $rF(P)$ is the set of $x\in C$ with $|x|_K=r$.

\begin{lem}\label{L:additive}
Let $K_1\oplus K_2$ be a free sum of co-convex sets 
$K_i=c(P_i)\subset C_i\subset \R^{n_i}$, for $i=1,2$. Then
\begin{enumerate}
\item[(a)] $F(P_1\oplus P_2)=\{((1-t)p_1,tp_2)\in C_1\times C_2 \ |\ p_i\in F(P_i), 0\leq t\leq 1\}$,
\item[(b)] $|x|_{K_1\oplus K_2}=|x_1|_{K_1}+|x_2|_{K_2}\,$ 
for any $x=(x_1,x_2)\in C_1\times C_2$.
\end{enumerate}
\end{lem}

\begin{pf} {\it (a)} First, by convexity of the $P_i$ we have
\begin{equation}\label{e:freesum}
P_1\oplus P_2=\{\left((1-s)v_1,sv_2\right)\in C_1\times C_2\ |\ v_i\in P_i, \, 0\leq s\leq 1\}.
\end{equation}

Pick $p_i\in F(P_i)$, for $i=1,2$, and consider  $p=\left((1-t)p_1,tp_2\right)$ for some $0\leq t\leq 1$. Let $\G_i$ be a bounded face of $P_i$ containing $p_i$ with inner normal $u_i$, and let $b_i=\min_{v_i\in P_i}\langle u_i,v_i\rangle=\langle u_i,p_i\rangle$. Note that $b_i> 0$ since $0\not\in \G_i$, so by rescaling the $u_i$ we may assume that $b_i=1$.
Put $u=(u_1,u_2)$. Then $\langle u,p\rangle=1$. On the other hand,
for any $v=\left((1-s)v_1,sv_2\right)\in P_1\oplus P_2$ we have 
$$\langle u,v\rangle=(1-s)\langle u_1,v_1\rangle+s\langle u_2,v_2\rangle\geq 1.$$
This shows that $p$ belongs to a bounded face of $P_1\oplus P_2$.

Conversely, if $p\in F(P_1\oplus P_2)$ then $\langle u,p\rangle=\min_{v\in P_1\oplus P_2}\langle u,v\rangle$  for some $u=(u_1,u_2)$. As above, by \re{freesum}, we have
$$\langle u,p\rangle=(1-t)\langle u_1,p_1\rangle+t\langle u_2,p_2\rangle\quad\text{for some }\ 0\leq t\leq 1.$$
Therefore, $\langle u_i,p_i\rangle=\min_{v_i\in P_i}\langle u_i,v_i\rangle$ for $i=1,2$, i.e. $p_i\in F(P_i)$.


{\it (b)} Let $r=|x|_{K_1\oplus K_2}$. Then $x\in rF(P_1\oplus P_2)$, hence, by {\it (a)}
$x=(x_1,x_2)=(r(1-t)p_1,rtp_2)$ for some $p_i\in F(P_i)$ and $0\leq t\leq 1$. 
This implies that $|x_1|_{K_1}= r(1-t)$ and $|x_2|_{K_2}= rt$ and so 
$$|x_1|_{K_1}+|x_2|_{K_2}= r=|x|_{K_1\oplus K_2}.$$

\end{pf}

The following lemma is an easy adaptation of the calculation given in the proof of Lemma 3.2 in \cite[p. 15]{RyaZva}.

\begin{lem}\label{L:integral} Let $K\subset C$ be a co-convex body. Then
$$\int_{C}e^{-|x|_K}\,dx=n!\vol_n(K)=\Vol_n(K).$$
\end{lem}

Now the above mentioned property of the free sum follows 
from the two lemmas and the Fubini theorem.

\begin{prop}\label{P:freesums}
Let $K_1\oplus K_2$ be a free sum of co-convex sets 
$K_i=c(P_i)\subset C_i\subset \R^{n_i}$, for $i=1,2$. Then
$$\Vol_{n_1+n_2}(K_1\oplus K_2)=\Vol_{n_1}(K_1)\Vol_{n_2}(K_2).$$
\end{prop}

\begin{pf} Indeed, by \rl{integral} and \rl{additive}, part (b)
$$\Vol_{n_1+n_2}(K_1\oplus K_2)=\int_{C_1\times C_2}e^{-|x|_{K_1\oplus K_2}}\,dx=\int_{C_1}e^{-|x_1|_{K_1}}\,dx_1
\int_{C_2}e^{-|x_2|_{K_2}}\,dx_2=\Vol_{n_1}(K_1)\Vol_{n_2}(K_2).$$
\end{pf}

Now let an ideal $I\subset R=k[x_1,\dots, x_n]_{(x_1,\dots, x_n)}$ be the sum of monomial ideals
whose sets of generators involve pairwise disjoint collections of variables. 
Then \rp{freesums} provides us with the following multiplicativity property of the $j$-multiplicity. 

\begin{thm}\label{T:product}
Assume that the set of the variables $\{x_1,\dots, x_n\}$ is partitioned into subsets 
$X_1,\dots, X_s$ and consider the ideal $I=I_1R+\dots+I_sR$ for some monomial 
ideals $I_k\subset R_k=k[X_k]_{(X_k)}$ for $k=1, \ldots, s$. Then
$$j(I)=j(I_1)\cdots j(I_s).$$
\end{thm}

\begin{pf} Let $C\subset\R^n$ be the cone over $F(I)$ and $P=P(I)\cap C$ as above. 
Then the $j$-multiplicity $j(I)$ equals the normalized volume of
the co-convex body $c(P)$, as in \re{coconvex}. Similarly, let
$C_k\subset\R^{n_k}$, where $n_k=|X_k|$, 
be the cone over $F(I_k)$ and  $P_k=P(I_k)\cap C_k$. Then 
$j(I_k)$ equals the normalized volume of $c(P_k)$.
On the other hand, $c(P)$ equals the free sum $c(P_1)\oplus\cdots\oplus c(P_s)$.
Therefore, by \rp{freesums} we have 
$$j(I)=\Vol_{n}(c(P))=\Vol_{n_1}(c(P_1))\cdots \Vol_{n_s}(c(P_s))=j(I_1)\cdots j(I_s).$$
\end{pf}

\begin{rem}
It would be interesting to give an algebraic proof of \rt{product}. For instance, using \rt{edgesubring} and \rt{cedgesubring} one may give an algebraic proof for the case of edge ideals of $m$-uniform hypergraphs with properly connected components. 
Moreover, using methods of commutative algebra we can show \rt{product} holds for arbitrary zero-dimensional ideals, or for  arbitrary homogenous ideals generated in the same degree. This leads us to believe that \rt{product} holds true even if the ideals involved are not monomial. These results will be addressed in a subsequent paper.
\end{rem}

\section{The $j$-multiplicity of edge ideals and volumes}\label{S:edgeidealvolumes}

Consider a hypergraph $G$ with the node set  $V(G)=\{x_1,\dots, x_n\}$ and the edge set $E(G)$. By definition, $E(G)$ consists of finitely many subsets of $V(G)$, called edges of $G$. We say $G$ is  {\it $m$-uniform} if each edge of $G$ has size $m$.
 Note that a {\it simple graph} is a 2-uniform hypergraph. By abuse of notation we let $k[x_1,\dots, x_n]$ be a polynomial ring generated by the $x_i$ as indeterminates over a field~$k$.  
  To every edge $\{x_{i_1},\ldots, x_{i_k}\}$  in $G$  we associate a square-free monomial $x_{i_1}\cdots x_{i_k}$ in the local ring $R=k[x_1,\dots, x_n]_{(x_1,\dots, x_n)}$. 
Then the {\it edge ideal of $G$} is 
$$I(G) =(x_{i_1}\cdots x_{i_k} \ | \  \{x_{i_1}, \ldots, x_{i_k} \} \in E(G))\subset R.$$ 
We denote the Newton polyhedron and the Newton polytope of $I(G)$
simply by $P(G)$ and $F(G)$, respectively. 
Following \cite{OH1998, V1998} we call $F(G)$ the {\it edge polytope} of $G$.

Assume $G$ is $m$-uniform. Then it can be readily seen that the monomials in $R$ associated to the edges of $G$ are the minimal generators of $I(G)$.
Note that $F(G)$ is the convex hull of some  lattice points in $\Z^n$ in which all entries are zero except for $m$ entries which are $1$. Thus,  
$F(G)$ lies in the hyperplane 
$$L=\{(z_1,\dots,z_n)\in\R^n \ | \ z_1+\cdots+z_n=m\},$$
and so the dimension of $F(G)$  is at most $n-1$. Therefore, the edge polytope $F(G)$ is the unique maximal compact face of $P(G)$, and if the dimension of $F(G)$ is exactly $n-1$,  then $F(G)$ is the unique compact facet of $P(G)$. 
Recall the formula in  \rt{JonathanJack} on the $j$-multiplicity of a monomial ideal and the volume.
For the edge ideal $I(G)$, there is only one term in the sum corresponding to $F(G)$ as the unique compact facet when the $j$-multiplicity is not zero. In this case, the volume of the pyramid 
$\Pyr(F(G))$ is computed by \re{pyrvolume} where the lattice distance $h(F(G))=m$. Therefore, we obtain the following result connecting the $j$-multiplicity to the volume of the edge polytope.
\begin{cor}\label{C:JonathanJack}
Let $G$ be an $m$-uniform hypergraph on $n$ nodes. Then $$j(I(G))=m \cdot \Vol_{n-1}(F(G)).$$
\end{cor}

Let $G$ be a hypergraph on $n$ nodes. If $G$ has an isolated node, then every generator of $I(G)$ will be missing at least one of the variables which makes $F(G)$ of dimension less than $n-1$. Therefore, $j(I(G))$ is zero. Similarly, if the number of edges of $G$ is less than the number of nodes, then $j(I(G))$ is zero. Therefore,

\begin{rem}\label{R:iso}
If $G$ is  a hypergraph  with an isolated node, or if the number of edges of $G$ is less than the number of nodes, then $j(I(G))=0$. Thus, for the rest of this paper we will assume that the hypergraphs in question do not have isolated nodes, and the number of edges of each connected component is at least the number of its nodes. 
\end{rem}

A hypergraph $G$ is called {\it connected} if for any two nodes $x_i, x_j \in V(G)$, there is a sequence of edges  in $E(G)$ such that $x_i$ and $x_j$ belong to the first and the last edges of the sequence respectively, and consecutive edges in  the sequence  have a common node. 
Let $G_1,\dots, G_c$ be the connected components of $G$. Then the edge ideal $I(G)$ is the sum of the extensions of the edge ideals $I(G_k)$ for $k=1, \ldots, c$ whose generators depend on  pairwise disjoint collections of variables. Therefore, by \rt{product} we obtain the following result.
\begin{prop}\label{P:product}
Let $G_1,\dots, G_c$ be the connected components of a hypergraph $G$. Then
$$j(I(G))=j(I(G_1))\cdots j(I(G_c)).$$
\end{prop}

Recall that by a result of Bivi{\`a}-Ausina \cite{AB}, for a monomial ideal the analytic spread equals one plus the  maximum of the dimensions of the compact faces of the Newton polyhedron. If $I(G)$ is the edge ideal of an $m$-uniform hypergraph $G$ on $n$ nodes and $e$ edges, then   $F(G)$ is the unique maximal compact face of the Newton polyhedron $P(G)$.  Therefore,
$$\ell(I(G))= 1+\dim F(G)=\rank M(G),$$ 
where $M(G)$ denotes  the $e\times n$ {\it incidence matrix} of $G$.
If $G$ is a simple graph, then $\rank M(G)$ is equal to $n-c_0$, where $c_0$ is the number of connected components of $G$ that contain no odd cycles, i.e. the number of bipartite components of $G$ \cite{Grossman}. Hence, 

\begin{rem}\label{R:spread}
If $I(G)$ is the edge ideal of an  $m$-uniform hypergraph $G$, then $\ell(I(G))$ is the rank of the incidence matrix of $G$. In particular, if $G$ is a simple graph on $n$ nodes, then $\ell(I(G)) = n-c_0$.
\end{rem}

Using \rr{jvanishing} and \rr{spread} we obtain the following characterization for  positivity of the $j$-multiplicity of edge ideals of simple graphs.

\begin{prop}\label{P:positivity}
If $G$ is a simple graph, then $j(I(G))\not =0$ if and only if all connected components of $G$  contain an odd cycle, that is they are non-bipartite. 
\end{prop}
In \rs{pecas} we generalize \rp{positivity} to $m$-uniform hypergraphs. 
If a simple connected graph has the same number of nodes as the number of edges, then it contains exactly one cycle, hence it is called {\it unicyclic}. Therefore, in a simple graph  the number of nodes is equal to the number of edges if and only if the connected components are unicyclic.  The following result computes the $j$-multiplicity of the edge ideals of simple graphs with unicyclic components.  In the following proof, $\tau_0$ stands for the maximum number of node-disjoint odd cycles in $G$, called {\it odd tulgeity} of $G$.

\begin{prop}\label{P:unicyclic} 
Let $G$ be a simple graph with $c$ connected components and $e=n$. If $j(I(G)) \not =0$, then
$j(I(G))= 2^{c}$.
In particular, if $G$ unicyclic, then $j(I(G))= 2$ when $G$ has an odd cycle, and it is zero otherwise.
\end{prop}
\begin{proof}
Since $e=n$, by \rp{positivity}  we obtain $j(I(G))\not = 0$ if and only if each connected component has exactly one odd cycle. Thus in this case, $\tau_0=c$. By \cite[Theorem 2.6]{Grossman}, the maximal minor of the incidence matrix $M(G)$ with maximum absolute value is $\pm 2^{\tau_0}$. But $M(G)$ is a square matrix in our case. Therefore,  the absolute value of $\det( M(G))$ is $2^{c}$.  Note that $\Pyr(F(G))$ is an $n$-simplex and  the vertices of $F(G)$ are exactly the rows of the incidence matrix $M(G)$.
Thus the normalized volume of $\Pyr(F(G))$ equals the absolute value of  $\det(M(G))$. Now the result follows from  \rt{JonathanJack}. 
\end{proof}

In \rp{muniformunicyclic} we prove an extension of \rp{unicyclic}  for $m$-uniform hypergraphs.

\begin{rem}
If $G$ is the {\it complete} $m$-uniform hypergraph on $n$ nodes, then  \rex{CompleteUniform} provides a closed formula for the $j$-multiplicity of $I(G)$ in terms of $m$ and $n$. 
\end{rem}

\section{The pivot equivalence relation and analytic spread}\label{S:pecas}

Let $G$ be an $m$-uniform hypergraph. By \rr{iso}, we will always assume that $G$ has no isolated nodes. Then $G$ is called {\it properly connected} if for any two edges $u, v$ in $E(G)$, there is a sequence of edges of $G$  starting with $u$ and ending with $v$, such that the intersection of consecutive edges  has size $m-1$. Note that simple connected graphs are properly connected. As in \cite{BK}, we define a relation $\approx$ on the set of nodes of $G$ by letting  $x_i \approx x_j$ if there is a subset $A \subset \{x_1, \ldots, x_n\} \setminus \{x_i, x_j\}$, such that $\{x_i\}\cup A$ and  $\{x_j\}\cup A$ are edges of $G$. Then we define an equivalence relation $\sim$ on the set of nodes of $G$  by declaring $x_i \sim x_j$  for two nodes  $x_i, x_j$ if there is a sequence of nodes $x_{i_1}, \ldots, x_{i_r}$ such that
$$
x_i = x_{i_1}\approx x_{i_2} \approx  \cdots \approx  x_{i_r} =x_j.
$$
Note that $x_i \sim x_i$ for $i=1, \ldots n$ as we assume $G$ has no isolated nodes. This equivalence relation is called {\it pivot equivalence} and it gives a partition of the nodes of $G$ into {\it pivot equivalence classes}.

\begin{prop}\label{P:HyperSpread}
Let $G$ be an $m$-uniform hypergraph on $n$ nodes in which the connected components are properly connected. Let $c$ be the number of connected components and $p$ be the number of pivot equivalence classes of $G$. Then
$$
\ell(I(G))=n-p+c.
$$ 
\end{prop}

\begin{proof}
Let $G_1, \ldots, G_c$ be the connected components of $G$. Since the $G_i$ are properly connected, then  by the main theorem of   \cite{BK} the rank of the incidence matrix of $G_i$ is $n_i-p_i+1$, where $n_i$ is the number of nodes and $p_i$ is the number of pivot equivalence classes in $G_i$.  Recall from \rr{spread} that the analytic spread of the edge ideal of $G$ can be computed as the rank of its incidence matrix, which is the sum of the ranks of the incidence matrices of the $G_i$. Hence the analytic spread of the edge ideal $I(G)$ is given by $\sum_{i=1}^c(n_i-p_i+1)$. Therefore, we may write $\ell(I(G))=n-p+c$. 
\end{proof}

Using \rr{jvanishing} and \rp{HyperSpread} we obtain the following characterization for  positivity of the $j$-multiplicity of edge ideals of $m$-uniform hypergraphs.

\begin{prop}\label{P:uniformpositivity}
Let $G$ be an $m$-uniform hypergraph in which the connected components are properly connected. Then
$j(I(G))\not = 0$ if and only if the nodes in each connected component of $G$  are pivot equivalent.
\end{prop}

If $G$ is a properly connected $m$-uniform hypergraph admitting pivot equivalence classes $V_1, \ldots, V_p$, then by
the first proposition of \cite{BK} there are fixed positive integers $b_1, \ldots, b_p$ such that each edge of $G$ contains exactly $b_i$ nodes from $V_i$ for $i=1, \ldots, p$. Hence $m=b_1 + \cdots + b_p \geq p$. Therefore, 

\begin{rem}
If $G$ is a properly connected $m$-uniform hypergraph, then $G$ has at most $m$ pivot equivalence classes. 
\end{rem}

For instance, if $G$ is a simple connected graph, then $G$ admits at most 2 pivot equivalence classes since two nodes are pivot equivalent if by definition they are connected by a {\it walk} of even length (see the definition of a walk in \rs{toricedgeideal}). Indeed, one may observe that $G$ admits only one pivot equivalence class if and only if $G$ contains an odd cycle. It follows that if $G$ is not connected, then $p=c+c_0$, where $c_0$ is the number of connected components of $G$ that contain no odd cycles. Hence $\ell(I(G))=n-p+c = n-c_0$ as in \rr{spread}.

\section{The $j$-multiplicity of edge ideals and edge subrings}\label{S:edgeidealedgesubring}

As in the previous section, let $I(G)\subset R=k[x_1,\dots, x_n]_{(x_1,\dots, x_n)}$ be the edge ideal of an $m$-uniform  hypergraph $G$ on $n$ nodes. Then the {\it edge subring} of $G$, denoted by $k[G]$,  is the subalgebra of $R$ generated by the edges of $G$. In other words,
$$
k[G]:= k[x_{i_1}\cdots x_{i_m} \ | \  \{x_{i_1}, \ldots, x_{i_m} \} \in E(G)] \subset R.
$$ 
Note that the edge subring of $G$ is a graded algebra generated in degree $m$, thus it can be regarded as a standard graded algebra by assigning degree 1 to its generators. The Hilbert-Samuel multiplicity of the edge subring with respect to this grading is denoted by $e(k[G])$.  Let $G$ be an $m$-uniform hypergraph on $n$ nodes with properly connected components. Then there is a natural  homogeneous  isomorphism between edge subring $k[G]$  and the special fiber ring of the edge ideal of $G$. Therefore, the Krull dimension of $k[G]$ is the analytic spread of $I(G)$. Hence by \rp{HyperSpread} we obtain,
\begin{rem}\label{R:DimEdge}
If $G$ is an $m$-uniform hypergraph with properly connected components, then
$$
\dim k[G] = n-p+c,
$$
where $n$ is the number of nodes,  $p$ is the number pivot equivalence classes  and $c$ is the number of connected components of $G$. 
\end{rem}

If $G$ is a simple graph on $n$ nodes in which all connected components contain an odd cycle, then  $\Vol_{n-1} (F(G))$ is equal to $2^{c-1}e(k[G])$  by \cite[Theorem 4.9]{GilVal}. Therefore, $j(I(G))=2^{c}e(k[G])$ by  \rc{JonathanJack}. The following result is an extension of this statement to $m$-uniform hypergraphs. Our proof is an algebraic argument that does not rely on the relation between multiplicities and volumes. 
We begin with the case that $G$ is properly connected.

\begin{thm}\label{T:edgesubring}
Let $G$ be a properly connected $m$-uniform hypergraph. If $j(I(G))\not = 0$,  then
$$j(I(G))= m \cdot e(k[G]).$$
\end{thm}
\begin{proof}
 Let $I$ denote the edge ideal of $G$ and assume $j(I) \not = 0$. Then $j(I)=e(\Gamma_{\m}(\cG))$ by definition, where $\cG$ is the associated graded ring of $R$ with respect to $I$, and $\m$ is the maximal ideal $(x_1,\dots, x_n)R$. By the associativity formula for multiplicities of graded modules over graded algebras,  
$$
e(\Gamma_{\m}(\cG)) = \sum \lambda ((\Gamma_{\m}(\cG))_P )\cdot e(\cG/P),
$$
where $\lambda$ denotes the length, and the sum runs over all minimal primes $P$ in the support of $\Gamma_{\m}(\cG)$ of dimension $n$. Recall the special fiber ring $\cG/\m\cG$ is isomorphic to $k[G]$, which is a domain. Therefore, $\m \cG$ is a prime ideal of $\cG$ of dimension $n$, since $\dim\, \cG/\m\cG = \ell(I) = n$ by  \rr{jvanishing}. Moreover, $\m \cG$ is in the support of $\Gamma_{\m}(\cG)$ and any prime ideal in the support of $\Gamma_{\m}(\cG)$ contains $\m \cG$  as some power of $\m\cG$ annihilates  $\Gamma_{\m}(\cG)$. Thus, $\m \cG$ is the only minimal prime in the support of $\Gamma_{\m}(\cG)$ of dimension $n$. Therefore, 
$$
j(I) = e(\Gamma_{\m}(\cG)) =  \lambda((\Gamma_{\m}(\cG))_{\m \cG}) \cdot e(\cG/\m \cG) = \lambda(\cG_{\m \cG}) \cdot e(k[G]).
$$
It remains to show that $\cG_{\m \cG}$ has length $m$. 
Let $\cR$ denote the Rees algebra of $I$, which is defined as
$$\cR = R[It]=R[x_{i_1}\cdots x_{i_m}t \ | \  \{x_{i_1}, \ldots, x_{i_m} \} \in E(G)].$$
Then $\cG=\cR/I \cR$ and so
$\cG_{\m \cG} \simeq \cR_{\, \m \cR}/I \cR_{\, \m \cR}$. We claim that the ideal $\m \cG_{\m \cG}  = \m \cR_{\, \m \cR}/I \cR_{\, \m \cR}$ is  principal. Since $G$ is properly connected and $j(I)$ is not zero, any two nodes  $x_i$ and $x_j$ in $G$  are pivot equivalent by \rp{uniformpositivity}. Then by  \rl{ReesPivot} below we have
 $(x_i)\cR_{\, \m \cR} = (x_j)\cR_{\, \m \cR}$. Thus, $\m \cR_{\, \m \cR} = (x_i) \cR_{\, \m \cR}$ for any node $x_i$ in $G$, which proves the claim.  Let $\{x_{i_1}, \ldots, x_{i_m}\}$ be an edge in $G$. Then  $$\m^m \cR_{\, \m \cR}  =(x_{i_1}) \cR_{\, \m \cR} \cdots (x_{i_m}) \cR_{\, \m \cR} =(x_{i_1}\cdots x_{i_m})\cR_{\, \m \cR} \subset  I \cR_{\, \m \cR} \subset  \m^m \cR_{\, \m \cR}.$$
Thus  $I \cR_{\, \m \cR} = \m^m \cR_{\, \m \cR} $. Hence, the principal ideal
$$\m^k \cG_{\m \cG}  = (\m^k +I) \cR_{\, \m \cR}/I \cR_{\, \m \cR}$$
 is zero if and only if $k \geq m$. Therefore,
$$
\lambda(\cG_{\m \cG} ) = \sum_{k=1}^{m}  \lambda(\m^{k-1} \cG_{\m \cG}/  \m^k \cG_{\m \cG})
= m.
$$
\end{proof}

\begin{lem}\label{L:ReesPivot}
Let $G$ be an $m$-uniform hypergraph. Let $\cR$ denote the Rees algebra of the edge ideal of $G$. If $x_i$ and $x_j$ are two nodes in $G$ that are pivot equivalent, then $(x_i) \cR_{\, \m \cR}=(x_j) \cR_{\, \m \cR}$.
\end{lem}
\begin{proof}
Note that if $\{x_{i_1}, \ldots, x_{i_m}\}$ is an edge in $G$, then $x_{i_1}\cdots x_{i_m} t \in \cR \setminus \m \cR$. 
Hence $x_{i_1}\cdots x_{i_m} t$ is invertible in $\cR_{\, \m \cR}$. If $x_{i}\approx x_{j}$, then there is a subset $A \subset \{x_1, \ldots, x_n\}\setminus \{x_i, x_j\}$, such that $\{x_i\}\cup A$ and  $\{x_j\}\cup A$ belong to $E(G)$. Write $A=\{ x_{p_1}, \ldots, x_{p_{m-1}} \}$.
Then  $x_{p_1}\cdots x_{p_{m-1}}x_it$ and $x_{p_1}\cdots x_{p_{m-1}}x_jt$  are invertible in the localization $\cR_{\, \m \cR}$. Therefore, 
$$
\frac{x_i}{1}= \frac{x_ix_{p_1}  \cdots x_{p_{m-1}} t}{ x_{p_1}\cdots x_{p_{m-1}}x_jt} \cdot \frac{x_j}{1},
$$
which implies that $(x_i)\cR_{\, \m \cR} = (x_j)\cR_{\, \m \cR}$.  If $x_i$ and $x_j$ are pivot equivalent, then there is a sequence of nodes $x_{i_1}, \ldots x_{i_r}$ such that
$$
x_i = x_{i_1}\approx x_{i_2} \approx  \cdots \approx  x_{i_r} =x_j.
$$
Hence by what we observed earlier,
$$(x_i) \cR_{\, \m \cR}=(x_{i_1}) \cR_{\, \m \cR} = \cdots = (x_{i_r}) \cR_{\, \m \cR} = (x_j) \cR_{\, \m \cR}.$$
\end{proof}

\begin{rem}
The converse of \rl{ReesPivot} is not true in general.  Indeed, if $(x_i) \cR_{\, \m \cR} = (x_j) \cR_{\, \m \cR}$, then one can show that there are two subsets of $E(G)$, with associated square-free monomials $\{m_1, \ldots m_s\}$ and $\{m'_1, \ldots m'_s\}$ in $I(G)$, such that 
\begin{equation}\label{convlem1}
x_i\, m_1 \cdots m_s = x_j\, m'_1 \cdots m'_s.
\end{equation}
But we cannot conclude that $x_i$ and $x_j$ are pivot equivalent. For example,  let $G$ be a $3$-uniform hypergraph with  $V(G)=\{x, y, z, w, x_1,x_2,x_3\}$ and $E(G)$  the triangles in the simplicial
complex illustrated in \rf{example2}.  Then one may directly verify that
\begin{equation}\label{convlem2}
w(xx_1x_2)(xx_1x_3)(xx_2x_3)(yzw)=x(xywt)(xzw)(x_1x_2x_3)^2.
\end{equation}
Note that the expression in each parenthesis in (\ref{convlem2}) corresponds to an edge in $G$, hence it is invertible in $\cR_{\, \m \cR}$ after multiplying by the variable $t$. Therefore, $(w)\cR_{\, \m \cR} = (x)\cR_{\, \m \cR}$. However, $x$ and $w$ are not pivot equivalent.  It would be interesting to find a combinatorial interpretation of (\ref{convlem1})  in graph-theoretical terms.

\begin{figure}[h]
\includegraphics*[height=5.2cm]{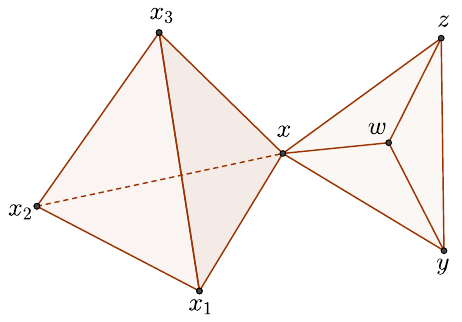}
\caption{The boundary of a tetrahedron attached to a union of three triangles.}
\label{F:example2}
\end{figure}
\end{rem}

Now we consider the case that $G$ has more than one properly connected component.

\begin{thm}\label{T:cedgesubring}
Let $G$ be an $m$-uniform hypergraph with properly connected components.  If $c$  is the number of components  and $j(I(G))$ is not zero,  then 
$$j(I(G))= m^c e(k[G]).$$
\end{thm}

\begin{proof}
Let  $G_1, \ldots, G_c$ denote the connected components of $G$. Then by \rp{product} and \rt{edgesubring} we obtain $j(I(G))= m^c e(k[G_1]) \cdots  e(k[G_c])$.  Therefore, the result follows from the main theorem of  \cite{North} which implies $e(k[G_1]) \cdots e(k[G_c])=e(k[G])$  since  $ k[G_1] \otimes_k \cdots \otimes_k k[G_c]  \simeq k[G]$. 
\end{proof}

Below we also sketch a direct proof of \rt{cedgesubring}  without using the multiplicativity formula in \rp{product} and the main result of \cite{North}.  

\begin{proof}
Let $\cG$ be the associated graded ring of $R$ with respect to the edge ideal $I$ of $G$. Then, as in the proof of  \rt{edgesubring},
$$
j(I)= \lambda(\cG_{\m \cG}) \cdot e(k[G]).
$$
We need to show that $\cG_{\m \cG}$ has length $m^c$. Recall that $\cG=\cR/I \cR$, where $\cR$ is the Rees algebra of $I$. Thus
$\cG_{\m \cG} \simeq \cR_{\, \m \cR}/I \cR_{\, \m \cR}$. 
Now let $X_k\subset\{x_1,\dots, x_n\}$ be the set of the nodes of the connected component $G_k$, 
so $\{x_1,\dots, x_n\}$ is the disjoint union of  $X_1,\dots, X_c$. After a possible relabeling of the nodes
we may assume that $x_k\in X_k$ for $k=1, \ldots,c$. Then  \rl{ReesPivot}
implies 
$(X_k) \cR_{\, \m \cR}= (x_k) \cR_{\, \m \cR}$ for  $k=1, \ldots, c$. Therefore,
$$\m \cG_{\m \cG}  = (x_1, \ldots, x_c) \cG_{\m \cG} .$$
Also $(x_k^m) \cG_{\m \cG}=(0)$ for all $k=1, \ldots, c$ as in the proof of  \rt{edgesubring}. Thus, by the pigeonhole principle
$$\m^{c(m-1)+1} \cG_{\m \cG}=(0).$$ 
Furthermore,  it can be readily seen that  for $i = 0, \ldots, c(m-1)$ the ideal 
$\m^i \cG_{\m \cG}$ is minimally generated by  monomials $x_1^{a_1} \cdots x_c^{a_c}$ of degree $i$ such that the $a_k$ are less than $m$. Therefore
$$
\lambda(\cG_{\m \cG} ) =\sum_{i=0}^{c(m-1)} \lambda(\m^{i} \cG_{\m \cG}/  \m^{i+1} \cG_{\m \cG}) 
$$
is the number of all monomials $x_1^{a_1} \cdots x_c^{a_c}$ such that the $a_k$ are less than $m$, which is $m^c$.
\end{proof}

\begin{ex}\label{ex:multipartite}
Let $G$ be the complete multipartite graph on $n$ nodes of type $(q_1, \ldots, q_k)$.  If $k$ is at least 3, then by \cite[Corollary 2.7]{Hibi} and  \rt{edgesubring} we obtain
$$
j(I(G))=2e(k[G])= 2^{n} - 2\sum_{i=1}^k \sum_{j=1}^{q_i} {n-1 \choose j-1}.
$$
\end{ex}

The following result is an immediate consequence of \rt{cedgesubring} and \rc{JonathanJack}
\begin{cor}\label{C:edgepolytope}
Let $G$ be an $m$-uniform hypergraph on $n$ nodes with properly connected components. If  $G$ has $c$ connected components
and $\Vol_{n-1}(F(G)) \not = 0$, equivalently, if the nodes in each connected component of $G$ are pivot equivalent, then
$$e(k[G])=m^{c-1}\Vol_{n-1}(F(G)).$$
\end{cor}

\begin{rem}
Note that in \rt{edgesubring}, if we do not assume $G$ is  properly connected then the statement fails, as the following example illustrates. Here $G$ is a connected $3$-uniform hypergraph with 
$V(G)=\{x_1,x_2,x_3,x_4,x_5,y_1,y_2,y_3\}$. The edge set $E(G)$ is given by the triangles in the simplicial
complex represented in \rf{example1}. 
\begin{figure}[h]
\includegraphics*[height=5.2cm]{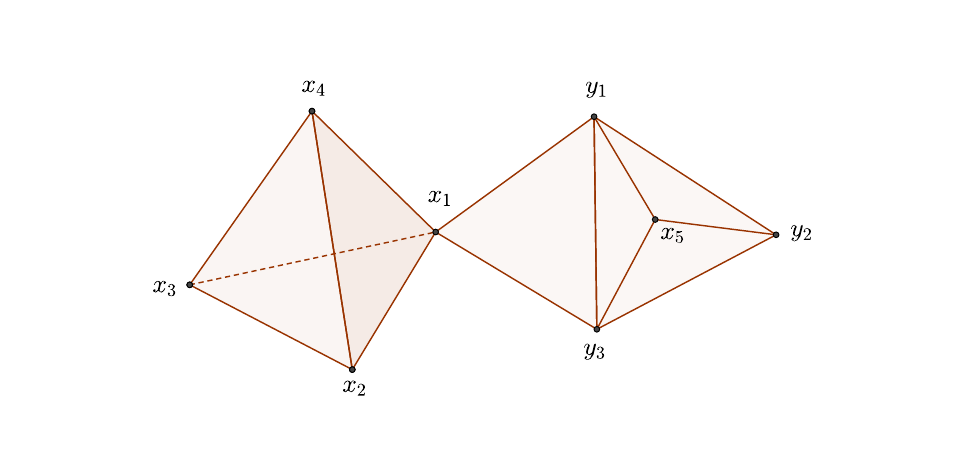}
\caption{The boundary of a tetrahedron attached to a union of four triangles.}
\label{F:example1}
\end{figure}
Note that $G$ has 8 nodes and 8 edges, and the incidence matrix $M(G)$ is
a square $8\times 8$ matrix of full rank. A simple calculation provides
$$j(G)=\Vol_8(\Pyr(F(G)))=\det M(G)=6.$$
 On the other hand, as in the proof of \rp{muniformunicyclic}, one can see that the edge ring $k[G]$ is isomorphic to a polynomial ring over a field,  and so $e(k[G]) = 1$, which shows that  \rt{edgesubring} fails for not properly connected hypergraphs. We can also calculate $j(G)$ directly as in the proof of \rt{edgesubring}. Recall that 
$$j(G)=e(\Gamma_{\m}(\cG)) = \lambda(\cG_{\m \cG}) \cdot e(k[G])=\lambda(\cG_{\m \cG}).$$
Let us show that the length of  $\cG_{\m \cG}\simeq \cR_{\, \m \cR}/I \cR_{\, \m \cR}$ is 6. 
First, note that $G$ has two pivot classes $\{x_1,x_2,x_3,x_4,x_5\}$ and $\{y_1,y_2,y_3\}$.
Then by  \rl{ReesPivot}  we have
$(x_i)\cR_{\, \m \cR} = (x_1)\cR_{\, \m \cR}$ for $i=1, \ldots 5$, and $(y_j)\cR_{\, \m \cR}= (y_1)\cR_{\, \m \cR}$ for  $j=1,2, 3$. Thus
$\m \cR_{\, \m \cR} = (x_1,y_1) \cR_{\, \m \cR}$. 
Using edges $\{x_1, y_1, y_3\}$ and $\{x_1, x_2, x_4\}$ we have $ x_2x_4 (x_1y_1y_3t)=y_1y_3 (x_1x_2x_4t) $. Hence we may write  
$$(x_1^2)\cR_{\, \m \cR} = (x_2x_4)\cR_{\, \m \cR}  =  (y_1y_3)\cR_{\, \m \cR} =  (y_1^2)\cR_{\, \m \cR}.$$
This implies that $\m^2 \cR_{\, \m \cR} = (x_1^2, x_1y_1) \cR_{\, \m \cR}$, $\m^3 \cR_{\, \m \cR} = (x_1^3, y_1^3) \cR_{\, \m \cR}$ and $\m^4 \cR_{\, \m \cR} = (x_1^4, x_1^3y_1)  \cR_{\, \m \cR}$.
Note that $(x_1^3) \cR_{\, \m \cR} = (x_1x_2x_3) \cR_{\, \m \cR} \subset I \cR_{\, \m \cR} $. Therefore, $\m^4\cG_{\m \cG}=(0)$ and $\m^{i} \cG_{\m \cG}/  \m^{i+1} \cG_{\m \cG}$ for $i=0,1,2, 3$ have bases
$\{1\}$, $\{x_1,y_1\}$, $\{x_1^2,x_1y_1\}$, and $\{y_1^3\}$, respectively. Thus,
$$
\lambda(\cG_{\m \cG} ) =\sum_{i=0}^{3} \lambda(\m^{i} \cG_{\m \cG}/  \m^{i+1} \cG_{\m \cG})=1+2+2+1=6.
$$

\end{rem}

\section{The $j$-multiplicity of edge ideals and toric edge ideals}\label{S:toricedgeideal}

Let $I(G)$ be the edge ideal of an $m$-uniform hypergraph  $G$ on $n$ nodes $x_1, \ldots, x_n$. 
As we mentioned in the previous section, the associated  edge subring  $k[G]$ can be regarded as a standard graded algebra over $k$. Therefore, we may define a homogeneous epimorphism of $k$-algebras
$$
\phi: S=k[T_{i_1 \cdots i_m} \ | \  \{x_{i_1}, \ldots, x_{i_m}\} \in E(G)] \to k[G],
$$
where the $T_{i_1 \cdots i_m}$ are indeterminates over $k$, by assigning $\phi(T_{i_1 \cdots i_m})=x_{i_1} \cdots x_{i_m}$ for $\{x_{i_1}, \ldots, x_{i_m}\} \in E(G)$. Thus one obtains a homogeneous  isomorphism $k[G] \simeq S/I_G$, where $I_G=\ker (\phi)$ is a homogeneous prime ideal  called the {\it toric edge ideal} of $G$. Indeed,  the ideal $I_G$ is generated by binomials, defining an affine toric variety \cite{Strumfels1}.

\begin{prop}\label{P:heighttoric}
Let $G$ be an $m$-uniform hypergraph on $n$ nodes with  properly connected components. Let $e$ denote the number of edges,
$p$  the number of pivot equivalence classes and $c$  the number of connected components of $G$. Then
$$
{\rm ht} \, I_G =e-n+p-c.
$$
\end{prop}
\begin{proof}
Recall that $\dim k[G] =\ell(I(G))$ by \rr{DimEdge}.  Thus, one can compute the height of the toric edge ideal of $G$ as
$$
{\rm ht} \, I_G=\dim S - \dim k[G] = e-\ell(I(G)).
$$
If all connected components of $G$  are properly connected then $\ell(I(G)) =n-p+c$ by  \rp{HyperSpread} and the result follows.
\end{proof}

Recall that if $j(I(G)) \not = 0$, then by  \rr{iso} the number of edges of $G$ is at least the number of nodes of $G$. The following result deals with the extremal case and extends \rp{unicyclic} to $m$-uniform hypergraphs.

\begin{prop}\label{P:muniformunicyclic}
Let $G$ be an $m$-uniform hypergraph  with properly connected components.
Assume the number of edges of $G$ is equal to the number of nodes of $G$.  If  $G$ has $c$ connected components and   $j(I(G)) \not = 0$,  then 
$$j(I(G))= m^{c}.$$
\end{prop}
\begin{proof}
Since all connected components of $G$ are properly connected and  $j(I(G)) \not = 0$, by \rp{uniformpositivity}, each connected component of $G$  admits only one pivot equivalence class. Then by \rp{heighttoric} the toric edge ideal $I_G$ has height zero. Thus $I_G$ is zero. Hence $k[G]$ is isomorphic to a polynomial ring over a field, and thus $e(k[G])=1$. Therefore, by \rt{cedgesubring} we obtain 
$$j(I(G))= m^{c}e(k[G])= m^{c}.$$
\end{proof}

\begin{ex}
If $G$ is the complete $(n-1)$-uniform hypergraph  on $n$ nodes, then $e=n$. In addition, $G$ is properly connected and has only one pivot equivalence class. Therefore, by \rp{muniformunicyclic} we obtain  $j(I(G))=n-1$, as in \rex{CompleteUniform}.
\end{ex}

Recall that a {\it walk} $w$ of length $s$ in  a simple graph $G$ is a sequence of edges of the form
$$\{x_{i_0}, x_{i_1}\}, \, \{x_{i_1}, x_{i_2}\},   \ldots,  \{x_{i_{s-1}}, x_{i_s}\}.$$
A walk  $w$ is called {\it closed} if the initial and the end nodes  $x_{i_0}, x_{i_s}$ are the same.  If $w$ is a closed walk of even length $2l$,  then we call $w$ a {\it monomial walk} and we define
$$T_w=T_{i_0i_1}T_{i_2i_3}\cdots T_{i_{2l-2}i_{2l-1}}-T_{i_1i_2}T_{i_3i_4}\cdots T_{i_{2l-1}i_{2l}} \in S,$$
which belongs to the toric edge ideal $I_G$.  Indeed, the toric edge ideal $I_G$ is generated by  binomials of the form $T_w$ associated to monomial walks in $G$ \cite{Vill1}. More generally, one may define monomial walks in an $m$-uniform hypergraph $G$ such that the toric edge ideal $I_G$ is generated by the associated binomials \cite{Petrovic}. We say a monomial walk $w$ is  {\it nontrivial} if $T_w \not = 0$, and   {\it minimal} if $T_w$ is irreducible. For example, if $G$ is unicyclic with an odd cycle, then  it does not admit a nontrivial  monomial walk, hence $I_G$ is zero as we observed in the proof of \rp{muniformunicyclic}. Two monomial walks $w$ and $w'$ are called {\it equivalent} if  $T_w = T_{w'}$. \\

A simple connected graph $G$  is called {\it bicyclic} if the number of edges is one more than the number of nodes. For instance, if $G$ is a simple graph  obtained by connecting two disjoint cycles with a path, then $G$ is a bicyclic graph known as a {\it bowtie} (\rf{bicycle1}). If $G$ consists of two cycles with a common node, then we regard it as a bowtie graph where the length of the path between the two cycles is zero. The following result computes the $j$-multiplicity of the edge ideals of bicyclic graphs.

\begin{figure}[h]
\includegraphics*[height=4cm]{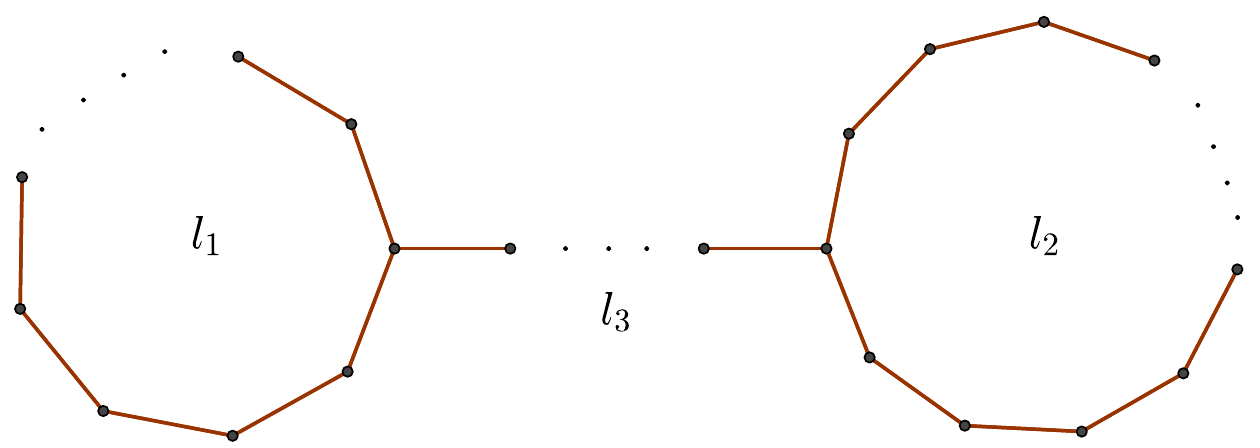}
\caption{A bicyclic graph of type 1.}
\label{F:bicycle1}
\end{figure}

\begin{prop}\label{P:bicyclic}
Let $G$ be an $m$-uniform hypergraph  with properly connected components. 
Assume the number of edges in $G$ is one more than the number of nodes and $G$ has $c$ connected components.
If $j(I(G)) \not = 0$,  then there is a unique nontrivial minimal monomial walk $w$ in $G$ up to equivalence.  Furthermore, if the length of $w$ is $2l$, then 
$$j(I(G))=m^cl.$$
In particular, if $G$ is a  bicyclic graph with an odd cycle, then $j(I(G))$ is the length of the unique nontrivial minimal monomial  walk in $G$.
\end{prop}

\begin{proof}
Recall that by \rp{uniformpositivity} $j(I(G))\not = 0$ if and only if  each connected component of $G$ contains only one pivot equivalence class.  Then we have ${\rm ht} \, I_G =e-n+p-c=1$  by \rp{heighttoric}.
Therefore, $I_G$ is a principal prime ideal generated by an irreducible homogeneous binomial $T_w$ corresponding to a unique minimal  monomial walk $w$ in $G$ up to equivalence. Hence,   we obtain $e(k[G])=e(S/I_G)=e(S/(T_w))=\deg T_w$. Thus, by \rt{cedgesubring} we conclude that 
$$j(I(G))=m^c \cdot e(k[G]) = m^c \cdot  \deg T_w.$$ 
Thus the result follows as the degree of  $T_w$ is half the length of the monomial walk $w$.
\end{proof}

\begin{ex}\label{ex:bowtie}
Let $G$ be a bicyclic graph, consisting of two cycles of lengths $l_1$ and $l_2$ connected by a path (\rf{bicycle1}) or attached along a path  of length $l_3$ 
 (\rf{bicycle2}).   If  both $l_1$ and $l_2$ are odd, then the length of the unique nontrivial minimal monomial walk in $G$ is $l_1+l_2+2l_3$ for the first type of graphs, and it is $l_1+l_2-2l_3$ for the second type of graphs. Thus,
$$j(I(G))=l_1+l_2\pm2l_3.$$
If   $l_1$ is odd and $l_2$ is even, then $j(I(G))=l_2$, and if  both $l_1$ and $l_2$ are even, then $j(I(G))=0$ by \rp{positivity}.
\end{ex}

 \begin{figure}[h]
\includegraphics*[height=4.7cm]{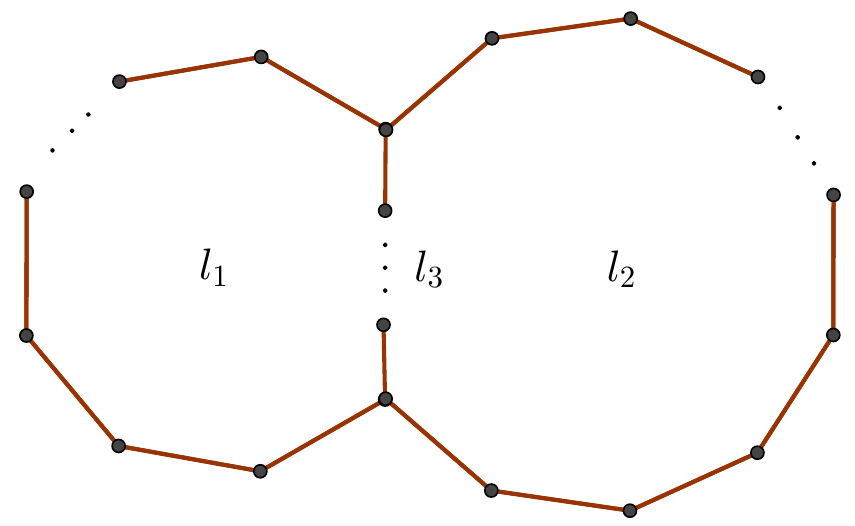}
\caption{A bicyclic graph of type 2.}
\label{F:bicycle2}
\end{figure}



One may also  obtain the following result as an immediate corollary of \rp{bicyclic}, \rp {muniformunicyclic} and \rp{product}.

\begin{cor}\label{C:unibi}
Let $G$  be a simple graph in which the connected components are unicyclic or bicyclic. If $j(I(G))$ is not zero,  then $$j(I(G))=2^{c}l_1 \cdots l_{k},$$
where $c$ is the number of connected components of $G$ and  the $l_{i}$ are half the length of the unique nontrivial minimal monomial walks in the  bicyclic connected components of $G$.  
\end{cor}

\begin{rem}
Note that the toric edge ideal of the graphs as in the statement of the \rc{unibi} are complete intersections. Let $G$ be an arbirtrary $m$-uniform hypergraph with complete intersection toric edge ideal $I_G$, generated by a regular sequence of binomials $T_{w_1}, \ldots, T_{w_s}$. Then 
$$e(k[G])=e(S/(T_{w_1}, \ldots, T_{w_s}))= \deg T_{w_1} \cdots \deg T_{w_s}.$$
Therefore, if   $G$ has properly connected components  and the $j$-multiplicity of the edge ideal of $G$ is not zero, then by \rt{cedgesubring} we obtain
$$
j(I(G))=m^c \cdot e(k[G]) =m^c   \deg T_{w_1} \cdots \deg T_{w_s} = m^{c}l_1 \cdots l_{s},
$$
where $l_{i}$ is half the length of the monomial walk $w_i$ for $i=1, \ldots, s$. In particular, we recover  \rc{unibi} without using \rp{product} and the volumes. For a study of simple graphs with complete intersection toric edge ideals, see \cite{BGR, GRV, TT}.
\end{rem}

\section{Inequalities on the $j$-multiplicity of edge ideals}\label{S:inequalityedgeideal}

In this section, we explore the relations between the $j$-multiplicity of  the edge ideals of hypergraphs and  their subhypergraphs and we obtain general bounds for the $j$-multiplicity of edge ideals. Let $G$ and $H$ be hypergraphs. Then $H$ is called a {\it subhypergraph} of $G$ if  $V(H)$ and $E(H)$ are subsets of $V(G)$ and $E(G)$, respectively. 
In \rt{inequality} below we prove a monotonicity property of the $j$-multiplicity, which will be useful in providing bounds for the $j$-multiplicity of edge ideals.
We start with the following geometric observation.

\begin{lem}\label{L:subpolytope} Let $A$ be any finite set of lattice points in $\R^n$ and 
$B\subset A$. Then the normalized volume of $\Conv(B)$ in the affine span of $B$ is no greater than the normalized volume of $\Conv(A)$ in the affine span of $A$.
\end{lem}

\begin{proof}
By induction, it is enough to assume that $|A|-|B|=1$. Also, by choosing coordinates
we may assume that the affine span of $A$ is $\R^n$. Let $A\setminus B=\{a\}$.
If the affine span of $B$ is also $\R^n$ then, clearly
$$\Vol_n(\Conv(B))\leq \Vol_n(\Conv(A)).$$
Otherwise, the affine span of $B$ is an affine hyperplane $L\subset \R^n$ and $\Conv(A)$
is the pyramid over $\Conv(B)$ with apex $a$. Then
$$\Vol_{n-1}(\Conv(B))\leq \Vol_n(\Conv(A))$$
follows from \re{pyrvolume-apex} since the lattice distance from the affine span of $B$ to $a$
is a positive integer.
\end{proof}

\begin{thm}\label{T:inequality}
Let $G$ be an $m$-uniform hypergraph. If $j(I(G))$ is not zero and $H$ is a subhypergraph of $G$, then
$$j(I(H))\leq j(I(G)).$$
\end{thm}
\begin{proof}
Let $A\subset\R^n$ consist of  the origin and the lattice points corresponding to the edges of $G$.
Then $j(I(G))=\Vol_n(\Conv(A))$ by \rt{JonathanJack}. The set of nodes $V(H)$ defines
a coordinate subspace of $\R^n$ which we identify with $\R^k$, where $k=|V(H)|$.
Similarly, let $B\subset\R^k$ consist of the origin and the lattice points corresponding to the edges of $H$, and, hence, $j(I(H))=\Vol_k(\Conv(B))$.  If the affine span of $B$ equals $\R^k$ then 
$j(I(H))\leq j(I(G))$ by \rl{subpolytope}. Otherwise, $j(I(H))=0$ and the inequality obviously holds. 
\end{proof}

\begin{rem} The above argument easily carries over to the case of arbitrary 
monomial ideals $I$ in $R=k[x_1,\dots, x_n]_{(x_1,\dots, x_n)}$ whose minimal monomial generators have 
exponents lying in a hyperplane (that is when $\dim F(I)<n$). 
Namely, if $\cB$ is a subset of the set of the minimal monomial generators of $I$
and $X\subset \{x_1,\dots, x_n\}$ is the set of variables appearing in $\cB$ then the ideal 
$J\subset k[X]_{(X)}$ generated by $\cB$ satisfies $j(J)\leq j(I)$.
Note that the condition $\dim F(I)<n$ is essential here as the following simple example shows.
If $I=\langle x^3,xy,y^3\rangle$ and $J=\langle x^3,y^3\rangle$ in $R=k[x,y]_{(x,y)}$ then $j(J)>j(I)$.
\end{rem}

\begin{cor}\label{C:completebound}
Let $G$ be an $m$-uniform hypergraph on $n$ nodes. Then $j(I(G))$ is bounded above by the $j$-multiplicity of the edge ideal  of the  complete $m$-uniform hypergraph on $n$ nodes mentioned in  \rex{CompleteUniform}. In particular if $G$ is a simple graph, then  $j(I(G))$ is at most $2^n -2n$.
\end{cor}

Let $G$ be a simple graph with {\it odd tulgeity} $\tau_0$, which is the maximum number of node-disjoint odd cycles in $G$. Let $H$ be a subgraph of $G$ consisting of $\tau_0$ node-disjoint odd cycles in $G$.  Then by \rp{unicyclic} or \rp{muniformunicyclic}, the $j$-multiplicity of $I(H)$ is $2^{\tau_0}$.  Therefore, if $I(G)$ has nonzero $j$-multiplicity, then $j(I(G)) \geq 2^{\tau_0}$ by \rt{inequality}. On the other hand, if $G$ is a multipartite graph of type $(q_1, \ldots, q_k)$, then by  \rt{inequality} $j(I(G))$ is bounded above by the $j$-multiplicity of the complete multipartite graph of type $(q_1, \ldots, q_k)$ as in  \rex{multipartite}. Therefore, we obtain the following corollary.

\begin{cor} \label{C:bounds}
Let $G$ be a simple multipartite graph of type $(q_1, \ldots, q_k)$ with $n$ nodes and odd tulgeity $\tau_0$. If the $j$-multiplicity of $I(G)$ is not zero, then
$$
2^{\tau_0} \leq j(I(G)) \leq 2^{n} - 2\sum_{i=1}^k \sum_{j=1}^{q_i} {n-1 \choose j-1}.
$$
\end{cor}

For a node $x$ in $G$, we let $G-x$ denote the subhypergraph of $G$ obtained by removing $x$ and the edges containing it from $G$. We say that $x$ is a {\it free node}  if it is contained in only one edge in $E(G)$. For simple graphs a free node is also known as a {\it whisker}. Recall that by \rt{inequality}, $j(I(G-x))\leq j(I(G))$ for any node $x$ in $G$. 
Below we note that equality holds for free nodes.

\begin{prop}\label{P:whisker}
Let $G$ be an $m$-uniform hypergraph containing a free node $x$. Then $$j(I(G))=j(I(G-x)).$$
\end{prop}

\begin{proof} If $x_i \in V(G)=\{x_1, \ldots, x_n \}$ is a free node, then removing $x_i$ and the corresponding edge from $G$ is equivalent to removing
the unique vertex  of  the edge polytope $F(G)$ with $z_i$-coordinate being~$1$. Note that $F(G)$ is a pyramid with apex at this vertex and base $F(G-x_i)$. Since the base lies in the hyperplane $z_i=0$,  the height of the pyramid is one. Therefore the normalized $(n-1)$-volume of $F(G)$ equals the normalized $(n-2)$-volume of the base $F(G-x_i)$. Then  by \rc{JonathanJack} we obtain
$$j(I(G))=m \cdot \Vol_{n-1}(F(G))=m \cdot \Vol_{n-2}(F(G-x_i))=j(I(G-x_i)).$$
\end{proof}

One could also prove \rp{whisker} algebraically for simple graphs using  toric edge ideals as follows. 

\begin{proof}
By \rp{product} we may assume $G$ is connected. We may further assume $G$ contains an odd cycle, otherwise the statement is trivilally true as both $j(I(G))$ and $j(I(G-x))$ are zero.
 Let $\alpha$ be the only edge in $E(G)$ containing $x$.  Then $\alpha$ is not part of any nontrivial minimal monomial walk in $G$. Therefore, if we write $k[G]\simeq S/I_G$ as in \rs{toricedgeideal}, then $\alpha$ corresponds to a variable $T_{\alpha}$ in $S$ not appearing in the generators of the toric edge ideal $I_G$. If we let $\bar{S} = S/(T_{\alpha})$ and consider $\alpha$ as an element in $k[G]$, then we have the following homogenous isomorphisms of graded $k$-algebras,
$$
k[G]/(\alpha) \simeq  S/(I_G+(T_{\alpha})) \simeq \bar{S}/I_{G-x} \simeq k[G-x].
$$
Therefore, using the homogenous short exact sequence 
$$
0 \rightarrow k[G](-1) \xrightarrow{\alpha} k[G] \rightarrow k[G]/(\alpha)  \simeq k[G-x] \rightarrow 0
$$
we obtain $e(k[G])=e(k[G-x])$. Now since both $G$ and $G-x$ are connected and contain an odd cycle, by  \rt{edgesubring} we conclude 
$$j(I(G))=2e(k[G])=2e(k[G-x])=j(I(G-x)).$$
\end{proof}

The following result gives a lower bound for the $j$-multiplicity of the edge ideal of an $m$-uniform hypergraph  in terms of the multiplicity of the associated edge subring.
\begin{prop}\label{R:inequalityedgesubring}
Let $G$ be an  $m$-uniform hypergraph with $c$ connected components, not necessarily properly connected. If $j(I(G))$ is not zero, then
$$j(I(G)) \geq m^c \cdot e(k[G]).$$
\end{prop}
\begin{proof}
If $G$ is a connected  $m$-uniform hypergraph, not necessarily properly connected,  then as in  the proof of  \rt{edgesubring} we have 
$$
j(I(G)) =  \lambda(\cG_{\m \cG}) \cdot e(k[G])
$$
when $j(I(G))$ is not zero. Note that $I \cR_{\, \m \cR} \subset  \m^m \cR_{\, \m \cR}$. Thus $\m^k \cG_{\m \cG}=  (\m^k +I) \cR_{\, \m \cR}/I \cR_{\, \m \cR}$  is not zero for   $k$ less than $m$. Hence,
$$
\lambda(\cG_{\m \cG} ) = \sum_{k\geq1}  \lambda(\m^{k-1} \cG_{\m \cG}/  \m^k \cG_{\m \cG}) \geq m.
$$
Therefore, $j(I(G))$ is greater than or equal to $m \cdot  e(k[G])$. 
If $G$ is not connected,  then the desired inequality follows from \rp{product} and the fact that the multiplicity of the edge subring is multiplicative over the connected components.

\end{proof}

Let $G$ be an $m$-uniform hypergraph with properly connected components. Assume the  toric edge ideal  $I_G$ is minimally generated  by binomials $T_{w_1}, \ldots, T_{w_s}$. For a description of the minimal generators of the toric edge ideals of simple graphs see \cite{RTT}.  Then as in \rs{toricedgeideal} we may represent the edge subring  $k[G]$ as $S/(T_{w_1}, \ldots, T_{w_s})$. Therefore, 
$$e(k[G])=e(S/(T_{w_1}, \ldots, T_{w_s}))\leq \deg T_{w_1} \cdots \deg T_{w_s}.$$
Hence, by  \rt{edgesubring} we obtain
$$
j(I(G))=m^c \cdot e(k[G]) \leq m^c   \deg T_{w_1} \cdots \deg T_{w_s}.
$$
Thus we have the following result.

\begin{prop} 
Let $G$ be an $m$-uniform hypergraph with properly connected components. Then
$$
j(I(G)) \leq  m^{c}l_1 \cdots l_{s},
$$
where the $l_{i}$ are half the length of the monomial walks in $G$ corresponding to a minimal generating set of $I_G$.
\end{prop} 

Let $G$ be a simple connected graph on $n$ nodes and $e$ edges, such that  the edge subring $k[G]$ is Cohen-Macaulay.  See for instance \cite{BHO} for a study of graphs with Cohen-Macaulay edge subring. Then Lemma 4.1 in \cite{Hibi2} states that  $\Vol_{n-1}(F(G))$  is at least  $e-n+1$ when $G$ is not bipartite. Therefore, by  \rc{JonathanJack} we obtain the following lower bound for the $j$-multiplicity of the edge ideal of $G$.

\begin{prop} 
Let $G$ be a simple connected graph on $n$ nodes and $e$ edges whose edge subring  is Cohen-Macaulay. If $ j(I(G))$ is not zero, then $$j(I(G)) \geq 2(e-n+1).$$
\end{prop}

\section{The $\eps$-multiplicity of edge ideals}\label{S:epsilonedgeideal}

We recall the notion of the $\eps$-multiplicity as introduced in \cite{KV} and \cite{UV}. Let $I$ be an arbitrary ideal in a Noetherian local ring $R$ with maximal ideal $\m$ and dimension $n$. Then the $\eps$-multiplicity of $I$ is defined as
$$ 
\eps(I)= n! \limsup_{k} \,  \frac{\lambda _R(\Gamma_{\m}(R/I^{k}))}{k^{n}} \in \R_{\geq 0}.
$$
Similar to the $j$-multiplicity, the $\eps$-multiplicity can be viewed as an extension of the Hilbert-Samuel multiplicity  to arbitrary ideals, for if $I$ is $\m$-primary, then $\Gamma_{\m}(R/I^{k})= R/I^{k}$, therefore $\eps(I) =  { e}(I)$.  However, the $\eps$-multiplicity exhibits a very different behavior than the $j$-multiplicity. For instance, the $j$-multiplicity is always a non-negative integer, while the $\eps$-multiplicity could be an irrational real number \cite{CHST}. In this section, we will compute the $\eps$-multiplicity of the edge ideal of cycles and complete hypergraphs, which further highlights the differences of the two invariants.
The vanishing of the $\eps$-multiplicity of an ideal is captured by the analytic spread of the ideal. Indeed, as in the case of $j$-multiplicity, the $\eps$-multiplicity of $I$ is not zero if and only if the analytic spread of $I$ is maximal  \cite{KV, UV}.  In particular,  by \rp{HyperSpread} we obtain the following result.

\begin{prop}\label{P:epsilon-vanishing}
If $G$ is an $m$-uniform hypergraph  with properly connected components, then  $\eps(I(G)) \not = 0$  if and only if  the nodes in each connected component of $G$ are pivot equivalent. 
Recall that for simple graphs, this condition means that each connected component contains an odd cycle.

\end{prop}

Let $I$ be a monomial ideal in  $R=k[x_1,\dots, x_n]_{(x_1,\dots, x_n)}$.  
Let $L_i\subset\R^n$ be the coordinate hyperplane defined by $z_i=0$ and $\pi_i:\R^n\to L_i$
the corresponding orthogonal projection. For the Newton polyhedron $P(I)$, define the following 
\begin{equation}\label{e:hatF}
\hat P(I)=\bigcap_{i=1}^n\pi_i^{-1}\left(\pi_i(P(I))\right),\quad \hat F(I)=\cl(\hat P(I)\setminus P(I)),
\end{equation}
where $\cl(K)$ denotes the closure of $K$ in $\R^n$.
The following theorem by Jeffries and Monta{\~n}o \cite[Theorem 5.1]{JM} gives an interpretation of the $\eps$-multiplicity of monomial ideals in terms of the volumes of the associated polytopes.

\begin{thm}\label{T:eps-volume} Let $I\subset R$ be a monomial ideal. Then $\eps(I)=\Vol_n(\hat F(I))$.
\end{thm}

Note that since $\hat P(I)\setminus P(I)$ is bounded, $P(I)$ and $\hat P(I)$ coincide outside of
a large enough ball. Therefore, $P(I)$ and $\hat P(I)$ have the same facet inequalities for
their unbounded facets. In particular, since $P(I)=F(I)+\R^n_{\geq 0}$, the inequalities $z_i\geq 0$
for $i=1,\ldots, n$ are among the facet inequalities for both $P(I)$ and $\hat P(I)$. 

\begin{prop}\label{P:epsiloncomplete} 
Let $G_{m,n}$ be the complete $m$-uniform hypergraph on $n$ nodes. Then
$$
\eps(I(G_{m,n}))=\frac{n-m}{n-1}A(n-1,m).
$$
In particular, for the complete simple graph $G_{2,n}$ and for the complete $(n-1)$-uniform hypergraph $G_{n-1,n}$ we obtain
$$\eps(I(G_{2,n}))=\frac{n-2}{n-1}(2^{n-1}-n),\quad\quad \eps(I(G_{n-1,n}))=\frac{1}{n-1}.$$
\end{prop}
\begin{proof} Denote $I_{m,n}=I(G_{m,n})$. Clearly, when $m=n$ we have $I_{n,n}=(x_1\cdots x_n)$ and $\eps(I_{n,n})=0$ which agrees with the formula in the statement. Thus we may assume that $m>n$. 
Let $P=P(I_{m,n})$ be the Newton polyhedron of $I_{m,n}$ and $F=F(I_{m,n})$ its compact facet. 
Recall from \rex{CompleteUniform} that $F$ is given by
$\sum_{j=1}^nz_j= m$. For every $i=1, \ldots, n$ the projection $\pi_i(P)$ equals $P(I_{m-1,n-1})$
embedded in the coordinate hyperplane $z_i=0$. This implies that $\pi_i^{-1}\pi_i(P)$ 
has a facet given by $\langle u_i,z\rangle\geq m-1$, where $u_i=-e_i+\sum_{j=1}^ne_j$. Therefore, 
$\hat P(I_{m,n})$ is given by the facet inequalities $\langle u_i,z\rangle\geq m-1$ and $z_i\geq 0$ for all $i=1, \ldots, n$. Since these facets
are unbounded, they are also the unbounded facets of $P$. This shows that 
$\hat F(I_{m,n})$ is a pyramid over $F$ with apex $a=\big(\frac{m-1}{n-1},\dots, \frac{m-1}{n-1}\big)$.
Consequently, by \rex{CompleteUniform} and equation \re{pyrvolume-apex} we obtain
\begin{equation}\label{e:CompleteUniform}
\eps(I_{m,n})=\Vol_n(\hat F(I_{m,n}))=\Big(m-\frac{n(m-1)}{n-1}\Big)\Vol_{n-1}(F)=\frac{n-m}{n-1}A(n-1,m).
\end{equation}
\end{proof}

\begin{prop}\label{P:epsiloncycle} 
Let $G$ be a cycle of length $n$. If $n$ is even, then $\eps(I(G))=0$.  If $n$ is odd, then
$$\eps(I(G))=\frac{2}{n+1}.$$
\end{prop}

\begin{proof} If $n$ is even then $\eps(I(G))=0$ by \rp{epsilon-vanishing}, so assume $n=2k+1$ for
$k \in\N$. To simplify notation we set $P=P(I)$, $F=F(I)$, and let $\hat P=\hat P(I)$ and 
$\hat F=\hat F(I)$ as defined in \re{hatF}.
By \rt{eps-volume}, $\eps(I(G))=\Vol_n(\hat F)$. In \rp{pyramid} below we show that $\hat F$ is
the pyramid over $F$ with apex $a=\big(\frac{1}{k+1},\dots, \frac{1}{k+1}\big)$. Since $F$ lies
in the hyperplane $\sum_{j=1}^nz_j=2$ and $\Vol_{n-1}(F)=1$ 
the equation \re{pyrvolume-apex} produces
$$\Vol_n(\hat F)=\Big(2-\frac{n}{k+1}\Big)\Vol_{n-1}(F)=\frac{2}{n+1}.$$
\end{proof}

To show that $\hat F$ is a pyramid over $F$ we first describe the  facet inequalities of $\hat P$ in \rl{facets} below.
Recall that the {\it circulant matrix $C_u$} generated by a vector $u=(u_0,\dots, u_{n-1})\in\R^n$
is the $n\times n$ matrix whose rows are obtained by the cyclic permutations of the entries of $u$.
The associated polynomial $f_u(t)=u_0+u_1t+\dots+u_{n-1}t^{n-1}$ of $C_u$ gives a formula for the rank of $C_u$ \cite[Proposition 1.1]{Ingleton}:
\begin{equation}\label{e:rankC}
\rk(C_u)=n-\deg\left(\gcd(t^{n}-1,f_u(t))\right).
\end{equation}

\begin{lem}\label{L:facets} The facets of $\hat P$ are defined by the inequalities $I_n\,z\geq 0$, $C_u\,z\geq {\bf 1}$, 
where $I_n$ is the identity matrix, ${\bf 1}$ is the vector of 1's, and $C_u$ is the circulant matrix 
generated by $u=e_1+\sum_{i=1}^k e_{2i}\in\R^n$, where $n=2k+1$. The same inequalities define the unbounded facets of 
 $P$.
\end{lem}

\begin{proof} 
First let us describe the primitive normals to the facets of $F_i=\pi_i(F)$. By definition,
$F$ is an $(n-1)$-simplex lying in the hyperplane $\sum_{j=1}^nz_j=2$ whose vertices are the
rows of the incidence matrix of the cycle $G$. Then $F_i$ is
an $(n-1)$-simplex lying in $L_i$ whose vertices are the rows of the incidence matrix of a ``graph" $G_i$ which 
is a cycle with omitted $i$-th node, so the rows corresponding to the edges with a missing node are two
standard basis vectors, see \rf{cycle-missing-node} for an example. 

\begin{figure}[h]
\includegraphics*[height=5.8cm]{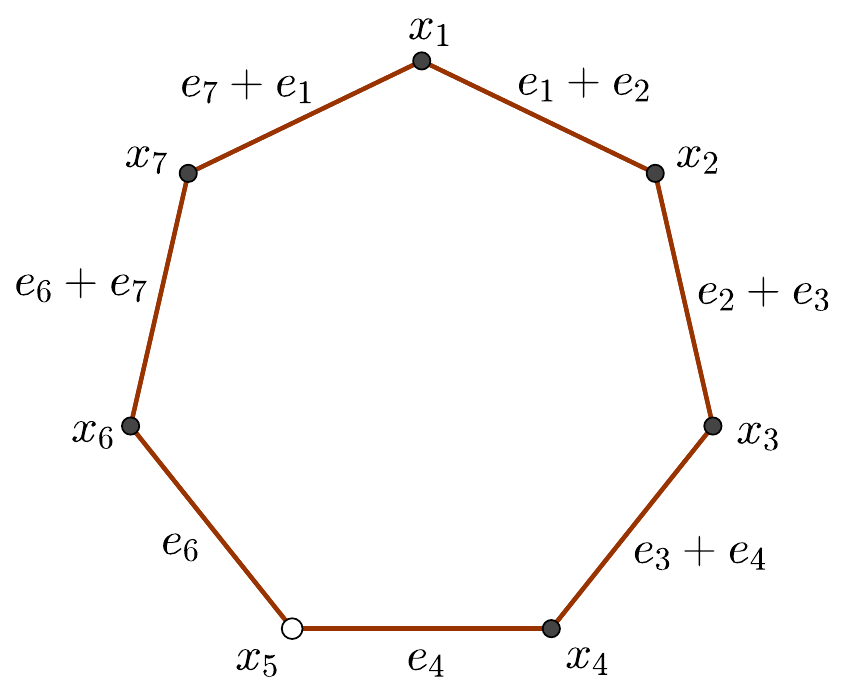}
\caption{A cycle with an omitted 5-th node.}
\label{F:cycle-missing-node}
\end{figure}

Since $F_i$ is a simplex, for every vertex $v\in F_i$ there is exactly one facet $F_i(v)$ not containing $v$. 
Here is a combinatorial way to produce a primitive normal to $F_i(v)$. (Note that 
its $i$-th entry can be arbitrary, so we may assume it is zero. Then it is unique up to sign.) 
Removing the edge from $G_i$  corresponding to $v$, we obtain a ``graph" $G_i(v)$. 
Place $0$ and $1$ at the nodes of $G_i(v)$
in an alternating way starting with the $0$ in $i$-th node and going both ways. This results
in a vector $u(v)\in\R^n$
 which is a primitive normal to $F_i(v)$. 
This process is illustrated in \rf{inner-normals} with $n=7$, $i=5$, and $v$ corresponding to the edge $\{x_1,x_2\}$.

\begin{figure}[h]
\includegraphics*[height=5.8cm]{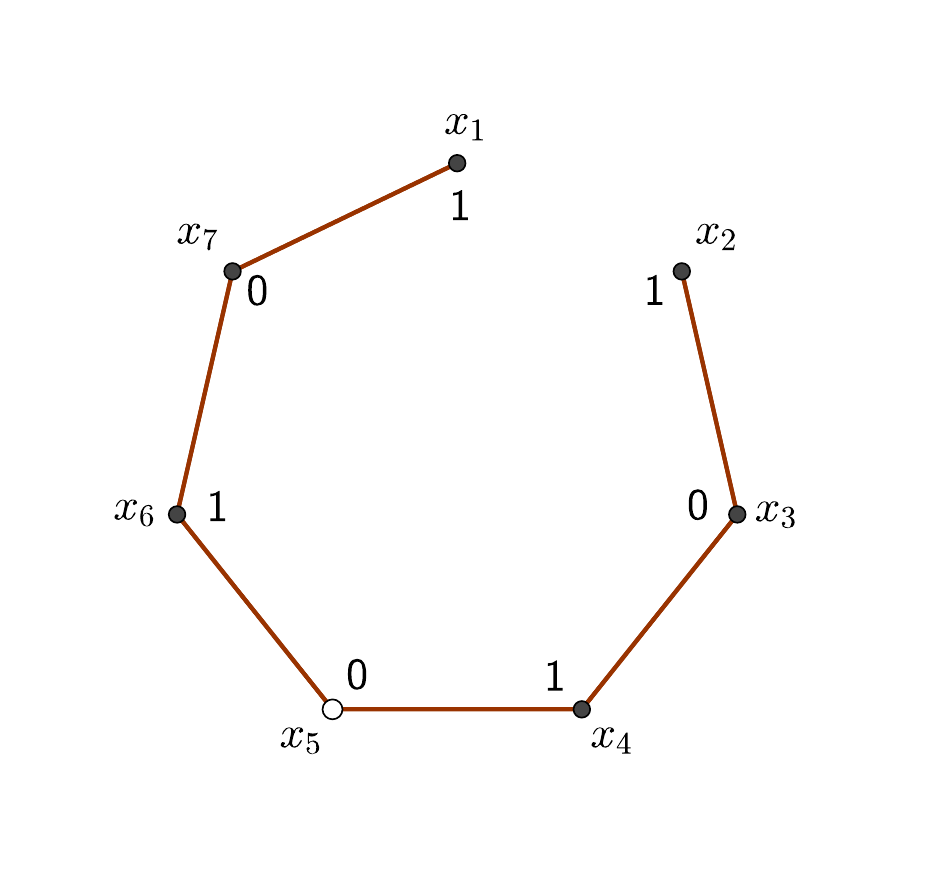}
\caption{The vector $u(v)=(1,1,0,1,0,1,0)$ is normal to $F_5(v)$ for $v=\{x_1,x_2\}$.}
\label{F:inner-normals}
\end{figure}

Indeed, $u$ is normal to $F_i(v)$ if and only if the linear function $\langle u,z\rangle$ takes the same value
at all vertices of $F_i$, but $v$. Assume for simplicity that $v$ corresponds to $\{x_1,x_2\}$ and $i=n=2k+1$. 
Then $v=e_1+e_2$ and the remaining vertices are $e_2+e_3,\dots, e_{2k-1}+e_{2k},e_{2k},e_1$.
Let $u=(u_1,\dots, u_{2k+1})$. Then $\langle u,z\rangle$ takes the same value on the remaining vertices if and only if
$$u_2+u_3=u_3+u_4=\dots=u_{2k-1}+u_{2k}=u_{2k}=u_1,$$
which implies $u_2=u_4=\dots=u_{2k}$ and $u_3=u_5=\dots=u_{2k-1}$, together with $u_{2k-1}=0$
and $u_{2k}=u_1$. Since $u$ is primitive, $u_1=u_2=u_4=\dots=u_{2k}=1$ which 
 justifies the combinatorial process of producing $u(v)$. The general case is similar.

Notice that the value of $\langle u(v),z\rangle$ at all vertices of $F_i$, but $v$ equals 1. Furthermore, its value at
$v$ equals the sum of the two values placed at the nodes of $v$.
These can be either both~1 or both~0. This shows that $u(v)$ is an inner normal to $\pi_i^{-1}(F_i)$ and, 
hence, to $\pi_i^{-1}(\pi_i(P))$ if and only if the two values are both 1.  Thus, the primitive inner
normals to the facets of  $\pi_i^{-1}(\pi_i(P))$ are vectors obtained by a cyclic
permutation of $(1,1,0,1,0\dots,1,0)$ and every such vector is the primitive inner normal to a
facet of $\pi_i^{-1}(\pi_i(P))$ for some $i$. Therefore, the facets
of $\hat P$ are given by $C_u\,z\geq {\bf 1}$ for $u=(1,1,0,1,0\dots,1,0)$, as stated.

Finally, we remark that all the facets of $\hat P$ are unbounded as the corresponding normals
have at least one coordinate equal zero. Thus, the same inequalities describe the unbounded facets of $P$.
\end{proof}

\begin{lem}\label{L:rank} Let  $C_u$ be the circulant matrix 
generated by $u=(1,1,0,1,0\dots,1,0)$ in $\R^n$ for $n=2k+1$. Then $\rk C_u=n$.
\end{lem}

\begin{pf} Let $f_u(t)=1+t+t^3+\dots+t^{2k-1}$ be the associated polynomial and let $g(t)=t^{n}-1$.
By \re{rankC}, $\rk C_u=n-\deg\left(\gcd(g(t),f_u(t))\right)$. Note that $(t^2-1)f_u(t)-g(t)=t(t-1)$. But 
neither $t=0$ nor $t=1$ is a root of $f(t)$, hence $\gcd(g(t),f_u(t))=1$ and the statement follows.
\end{pf}

\begin{prop}\label{P:pyramid}
The polytope $\hat F$ is the pyramid over $F$ with apex at $a=\big(\frac{1}{k+1},\dots, \frac{1}{k+1}\big)$.
\end{prop}
\begin{proof} 
Recall that $F$ is the unique compact facet of $P$ corresponding to the inequality $\sum_{j=1}^nz_j\geq 2$. 
Since $\hat F=\cl(\hat P\setminus P)$ lies in the other half space and the remaining facets inequalities for $\hat P$ and $P$ are the same, we conclude that $\hat F$ is given by  $C_u\,z\geq {\bf 1}$ and $\sum_{j=1}^nz_j\leq 2$. (One can see that the 
inequalities $I_n\,z\geq 0$ are redundant. Indeed, given $1\leq i\leq n$, add the two inequalities in 
$C_u\,z\geq {\bf 1}$ with $1$'s at the $i$-th and at the two adjacent places to obtain $z_i+2\geq z_i+\sum_{j=1}^nz_j\geq 2$,
which implies $z_i\geq 0$.) By \rl{rank}, $a=\big(\frac{1}{k+1},\dots, \frac{1}{k+1}\big)$ is the unique solution
to $C_u\,z= {\bf 1}$ which implies that $\hat F$ is the pyramid over $F$ with apex $a$.

\end{proof}

\begin{rem}\label{R:epsilonremarks}
Unlike the  $j$-multiplicity in \rp{product},
 the $\eps$-multiplicity of edge ideals  is not multiplicative over the connected components of a graph.
For instance, if $G$ is the disjoint union of a 3-cycle and a 5-cycle, then by direct computation using \rt{eps-volume} the $\eps$-multiplicity of the edge ideal of $G$ is $\frac{4}{9}$, while by \rp{epsiloncycle}  
the $\eps$-multiplicity of the edge ideals  of the 3-cycle and the 5-cycle are $\frac{1}{2}$ and $\frac{1}{3}$ respectively. Furthermore, in contrast to \rp{whisker} for $j$-multiplicity, the $\eps$-multiplicity is not preserved after removal of a free node. For example, if $G$ is a 3-cycle with a path of length 2 attached to one of its nodes, then the $\eps$-multiplicity of
$I(G)$  is indeed $\frac{1}{3}$, while after removing the free node the $\eps$-multiplicity of the edge ideal is $\frac{1}{2}$. This example also shows that the $\eps$-multiplicity may increase if we pass to a subgraph. Therefore \rt{inequality} does not hold true for the $\eps$-multiplicity of edge ideals. However, since the $\eps$-multiplicity is less than or equal to the $j$-multiplicity for an arbitrary ideal \cite{UV}, the upper bounds in \rc{completebound} and \rc{bounds} are valid for the $\eps$-multiplicity of the edge ideals as well. 
\end{rem}

\section*{Acknowledgment}

Evidence for this work was provided
by many computations done using {\tt Macaulay2}, by Dan
Grayson and Mike Stillman \cite{M2} and {\tt polymake}, by Ewgenij Gawrilow and Michael Joswig \cite{polymake}. We are grateful to Artem Zvavitch for pointing out the results on the volume of free sums in \cite[p. 15]{RyaZva} and for fruitful discussions. Finally, we are thankful to anonymous referees for suggestions on improving the results of \rt{product} and \rt{inequality}, as well as instructive comments which helped with the exposition.



\end{document}